\documentclass{amsart}
\usepackage{amsmath}
\usepackage{amsfonts}
\usepackage{amssymb}
\usepackage{amsthm}
\usepackage{graphicx}
\usepackage{wasysym}
\usepackage{listings}
\usepackage{mathrsfs}
\usepackage{tikz}
\theoremstyle{plain}% default
\newtheorem{thm}{Theorem}[section]
\newtheorem{lem}[thm]{Lemma}
\newtheorem{cor}[thm]{Corollary}
\newtheorem{prop}[thm]{Proposition}

\theoremstyle{remark}
\newtheorem{rem}[thm]{Remark}
\newtheorem{note}[thm]{Note}

\theoremstyle{definition}
\newtheorem{defn}[thm]{Definition}
\newtheorem{nota}[thm]{Notation}
\newtheorem{exmp}[thm]{Example}

\numberwithin{equation}{section}

\begin{document}

\title{Uniform extensions of layered semi-fields}
\author{Tal Perri}%

\begin{abstract}
In this paper we introduce a canonical method of constructing simple uniform semifield extensions of uniform layered semifields introduced in \cite{Layered}.
Our construction includes a decomposition of a uniform extension of a uniformly layered (uniform) semifield to the bipotent semifield extension of its $\nu$-values semifield and a cancellative semifields extension of its layers (sorting) semifield. We give a characterization of these two types of semifields extensions in the first two sections of the paper. The third section glues the pieces together to form a theory for a uniform extension of a uniformly layered semifield.
\end{abstract}

\maketitle

\section{Overview}

Consider the polynomial semiring $\mathbb{H}[x]$, where $\mathbb{H}=(\mathcal{G}(\mathbb{H}),\mathcal{L}_{\mathbb{H}})$
is a uniform layered semifield. Although $\mathbb{H}[x]$ is not a bipotent nor a (additively) cancellative semiring, for a given $a \in \mathbb{D}$ where $\mathbb{D}$ is a domain extending $\mathbb{H}$, the substitution $x=a$ gives rise to a much simpler structure on $\mathbb{H}[a]= \{ f(a) \ : \ f \in \mathbb{H}[x] \}$ which we model as an algebraic structure composed of a (additively) cancellative semiring and a bipotent semiring. Specifically for each $\nu$-value there is a layer fibre which is a (additively) cancellative semiring. In the last section we specify the set of elements $a \in \mathbb{D}$ for which $\mathbb{H}[a]$ is a uniform layered semifield. Moreover, for any $a \in D$ we build the minimal uniform layered semifield containing it as a composition of a pair of specific extensions called `pure extensions'.\\
Finally, as we show that a uniform layered extension can be decomposed as a pair of a (additively) cancellative and a bipotent affine extension, the complete description for both of these later cases made in the first two sections of this paper, completes the picture of simple uniform layered extensions.\\
Though used for the layered extension construction, the sections concerning bipotent and cancellative semifield extensions contain general results.
One of the models to which our construction applies to is that of uniform $\mathbb{Q}$-layered semifields.

\ \\
\section{Semifields and extensions of semifields}\label{section:SemifieldsPreliminaries}

\ \\

\begin{defn}\label{defn_affine_ext}
Let $\mathbb{H}$ be a semifield, and let $\mathbb{D}$ be a semiring extending $\mathbb{H}$.
We say that $\mathbb{D}$ is generated by a subset $A \subset \mathbb{D}$ over $\mathbb{H}$ if every element $a \in \mathbb{D}$ is of the form $\sum_{i=1}^{n}\alpha_i \prod_{j=1}^{m}a_{i,j}^{k_{i,j}}$ with $a_{i,j} \in A $ and $k_{i,j} \in \mathbb{N}$. $\mathbb{D}$ is said to be \emph{affine} over $\mathbb{H}$, or an \emph{affine extension} of $\mathbb{H}$, if $A$ is finite.

If $\mathbb{D}$ is affine over $\mathbb{H}$, we denote $\mathbb{D} = \mathbb{H}[a_1,...,a_n]$ where $\{a_1,...,a_n\}$ is a set of generators of $\mathbb{D}$ over $\mathbb{H}$. Namely, $\mathbb{\mathbb{D}} = \{ f(a_1,...,a_n) \ : \ f \in \mathbb{H}[x_1,...,x_n] \}$ where $\mathbb{H}[x_1,...,x_n]$ is the polynomial semiring with coefficients in $\mathbb{H}$.
\end{defn}

\begin{defn}
We say that a semiring $\mathbb{H}$ is a \emph{domain} when $\mathbb{H}$ is multiplicatively \linebreak cancellative.
\end{defn}

\begin{note}
In what follows, we refer to a domain semiring just as a `domain'.
\end{note}

\begin{defn}
Let $\mathbb{D}$ be an domain extending a semifield $\mathbb{H}$. The \emph{semifield of fractions}
of $D$ is defined to be
$$Frac(\mathbb{D}) = \left\{\frac{d_1}{d_2} \ : \ d_1, d_2 \in \mathbb{D}, \ d_2 \neq 0 \right\}.$$
If $\mathbb{D}$ is a semifield then $Frac(\mathbb{D})=\mathbb{D}$.\\
\begin{flushleft}If $\mathbb{D} = \mathbb{H}[a_1,...,a_n]$ is affine over $\mathbb{H}$, then \end{flushleft}
 $$Frac(\mathbb{D}) = \left\{ \frac{f(a_1,...,a_n)}{g(a_1,...,a_n)} \ : \ f,g \in \mathbb{H}[x_1,...,x_n], \ g(a_1,...,a_n) \neq 0  \right\} $$ where $\mathbb{H}[x_1,...,x_n]$ is the polynomial semiring with coefficients in $\mathbb{H}$. In this special case $Frac(\mathbb{D})$ is denoted as $\mathbb{H}(a_1,...,a_n)$.
\end{defn}

\begin{defn}\label{defn_simple_semifield_extension}
Let $\mathbb{S}$ be a semifield extending a given semifield $\mathbb{H}$.
$\mathbb{S}$ is said to be a \emph{simple} extension of $\mathbb{H}$ if there exists an element $d \in \mathbb{S}$ that generates $\mathbb{S}$ as a semifield over $\mathbb{H}$, i.e. $$\mathbb{S} = \left\{ \frac{f(d)}{g(d)} \ : \ f,g \in \mathbb{H}[x], \ g(d) \neq 0 \right\}.$$
\end{defn}

\ \\
\section{Bipotent extensions} \label{section:BipotentExtensions}  \ \\
%
%\title{Bipotent extensions}
%\maketitle
%
%\begin{defn}
% We define a \emph{bipotent semiring} $\mathbb{H}$ to be a semiring with bipotent addition, i.e.
% $ \alpha + \beta  \in  \{\alpha, \beta \}  $ for any $\alpha, \beta \in \mathbb{H}$.
% We say that $\mathbb{H}$ is a \emph{bipotent pre-domain} when there are no zero divisors in $\mathbb{H}$.
% When $\mathbb{H}$ is a semifield, i.e. every nonzero element of $\mathbb{H}$ is invertible with respect to multiplication, we say that $\mathbb{H}$ is a \emph{bipotent semifield}.
%\end{defn}
%

Recall that  a \emph{bipotent semiring} $\mathbb{H}$ is a semiring satisfying
$ \alpha + \beta  \in  \{\alpha, \beta \}  $ for any $\alpha, \beta \in \mathbb{H}$.

%$\mathbb{H}$ is a \emph{bipotent domain} when there are no zero divisors in $\mathbb{H}$.

\begin{rem}\label{rem_affine_bipotent}
Let $\mathbb{H}$ be a bipotent semiring, generated as a semiring by a proper subset $A \subset \mathbb{H}$. For any $a \in \mathbb{H}$, then since multiplication is distributive over addition, we can write
$$a = \sum_{i=1}^{n}\prod_{j=1}^{m} a_{i,j}^{k_{i,j}}$$ with some $a_{i,j} \in A $ and $k_{i,j} \in \mathbb{N}$. Since addition is bipotent the last expression reduces to an expression $\prod_{j=1}^{m} a_{i_0,j}^{k_{i_0,j}}$ with $i_0 \in \{1,...,n\}$. Thus any element of $\mathbb{H}$ is a finite product of elements in $A$. Consequently, $\mathbb{H}$ can be thought of as a multiplicative ordered monoid generated by $A$.
\end{rem}

\medbreak

In view of Definition \ref{defn_affine_ext} we have the following

\begin{rem}
If $\mathbb{D}$ is bipotent (thus so is $\mathbb{H}$), then by Remark \ref{rem_affine_bipotent}, $\mathbb{D}$ is generated by $A \subset \mathbb{D}$ over $\mathbb{H}$ if every element $a \in \mathbb{D}$ is of the form $\alpha \prod_{j=1}^{m}a_{j}^{k_{j}}$ with $a_{j} \in A $ and $k_{j} \in \mathbb{N}$. $\mathbb{D}$ is an affine extension if $A$ is finite, in which case, if $A = \{a_1,...,a_n\}$ is a set of generators of $\mathbb{D}$ over $\mathbb{H}$ then
$$\mathbb{D} = \mathbb{H}[a_1,...,a_n] = \{ f(a_1,...,a_n) \ : \ f \in \mathbb{H}[x_1,...,x_n] \ \text{is a monomial} \}.$$
Consequently, the semifield of fractions is of the form  \\

\begin{flushleft}$ Frac(D) = \mathbb{H}(a_1,...,a_n)$ \end{flushleft}
\begin{align}
= \left\{ \frac{m_1(a_1,...,a_n)}{m_2(a_1,...,a_n)} \ : \ m_1,m_2 \in \mathbb{H}[x_1,...,x_n] \ \text{ are monomials}, \  m_2(a_1,...,a_n) \neq 0 \right\} \nonumber \\
= \left\{\alpha \prod_{i=1}^{n}a_{i}^{k_i} \ : \ \alpha \in \mathbb{H}, \ k_i \in \mathbb{Z} \right\} \nonumber
\end{align}
where $\mathbb{H}[x_1,...,x_n]$ is the polynomial semiring with coefficients in $\mathbb{H}$.
\end{rem}

\bigbreak
\begin{defn}
If $\mathbb{D}$ is a bipotent semifield extending $\mathbb{H}$,
then we say that $\mathbb{D}$ is a \emph{bipotent extension} of $\mathbb{H}$.
\end{defn}

\begin{rem}\label{rem_quotient_monoid_stracure}
Let $\mathbb{H}$ be a bipotent semifield. Let $\mathbb{D}$ be a bipotent domain extending $\mathbb{H}$.
Since $\mathbb{H}$ is a semifield, $\mathbb{H}^{\ast}= \mathbb{H} \setminus \{ 0 \}$ is a multiplicative normal subgroup of $\mathbb{D}$ (which is commutative with respect to multiplication). Thus, the quotient monoid $\tilde{\mathbb{D}} = \mathbb{D}/{\mathbb{H}^{\ast}}$ is well-defined. Note that the operation of addition of $\mathbb{D}$ is not induced on~$\tilde{\mathbb{D}}$.\\
Equivalently, we can define the following relation on $\mathbb{D}$:\\

For every $a,b \in \mathbb{D}$
$$a \sim_{\mathbb{H}} b \Leftrightarrow  \exists  \alpha \neq 0 \in \mathbb{H} : a = \alpha b.$$

If $A \subset \mathbb{D}$ is a generating set of $\mathbb{D}$ over $\mathbb{H}$ then

$$\tilde{\mathbb{D}} = \mathbb{D}/\mathbb{H}^{\ast} = \left\{ \prod_{i=1}^{n}[a_i]^{k_i} \ : \ a_i \in A , \ k_i \in \mathbb{N} \right\} \cup \{0\} = \langle A \rangle \cup \{0\}$$
where $[  a  ]$ denotes the $\mathbb{H}^{\ast}$-coset of $a \in \mathbb{D}$ and $\langle A \rangle$ denotes the monoid generated by $A/ \mathbb{H}^{\ast} = \{[a] \ : \ a \in A \}$.\\
\end{rem}

\begin{rem}\label{rem_quotient_group_stracure}
In the special case of Remark \ref{rem_quotient_monoid_stracure} in which $\mathbb{D}$ is a semifield, we get that $\tilde{\mathbb{D}} = \mathbb{D}/{\mathbb{H}^{\ast}}$ is an abelian group.
\end{rem}

\begin{rem}\label{construction1}
Assume $\mathbb{D}= \mathbb{H}[a_1,...,a_n]$ is an affine bipotent semifield extending $\mathbb{H}$.  Since $\mathbb{D}$ is a semifield, we have that  $\mathbb{D}^{\ast} =\mathbb{D} \setminus \{ 0 \}$ is a multiplicative group, and thus so is  $\tilde{\mathbb{D}^{\ast}} = \mathbb{D}^{\ast}/{\mathbb{H}^{\ast}}$.  The set $\{ a_1,...,a_n \}$ generates $\mathbb{D}^{\ast}$ multiplicatively over $\mathbb{H}$. Thus $\left\{[a_1],...,[a_n] \right\}$ generates $\tilde{\mathbb{D}^{\ast}}$. So, we have that
$$\tilde{\mathbb{D}} = \tilde{\mathbb{D}^{\ast}} \cup \{ 0 \} = \left\langle [a_1],...,[a_n] \right\rangle \cup \{ 0 \}$$
where $\langle [a_1],...,[a_n] \rangle$ is the abelian group generated by $\left\{ [a_1],...,[a_n] \right\}.$\\

Denote $G = \langle [a_1],...,[a_n] \rangle$. By the fundamental theorem for finitely generated abelian groups, we can choose $[b_1],...,[b_t],[c_{t+1}],...[c_{m}] \in G$ such that
$$G = FT = \prod_{i=1}^{t} \langle [b_i] \rangle \prod_{i=t+1}^{m}  \langle [c_i] \rangle$$ where $F = \prod_{i=1}^{t} \langle [b_i] \rangle$ is free of rank $t$, and $T = Tor(G) = \prod_{i=t+1}^{m}  \langle [c_i] \rangle$ is the torsion subgroup of $G$ where $[c_i]^{n_i} = 1$ for appropriate natural numbers  $n_i \geq 2$ such that $n_{i+1} | n_{i}$ for $t \leq i \leq m-1$.\\
Moreover, the following hold for the elements in $\left\{[b_1],...,[b_t],[c_{t+1}],...[c_{m}]\right\}$:
$$\langle [b_j] \rangle \cap \langle \{[b_i] \ : \ i \in \{ 1,...,t \} \setminus \{ j \} \} \cup \{[c_i] \ : \ i=t+1,...,m \}  \rangle = [1]$$
for any $j = 1,...,t$,  and
$$\langle [c_k] \rangle \cap \langle \{ [b_i] \ : \ i = 1,...,t \} \cup \{ [c_i] \ : \ i \in \{ t+1,...,m \} \setminus \{ k \} \}  \rangle = [1] $$
for any $k = t+1, ..., m$.
\end{rem}

\begin{defn}
Let $\mathbb{H}$ be a semifield and let $\mathbb{D}$ be a bipotent domain extending $\mathbb{H}$.
The set $\{a_1, a_2,...,a_n \} \subset \mathbb{D}$ is said to be \emph{divisibly dependent} over $\mathbb{H}$ if there exist distinct monomials $m_1,m_2 \in \mathbb{H}[x_1,...,x_n]$ such that $m_1(a_1,...,a_n) = m_2(a_1,...,a_n)$. Otherwise $\{ a_1,...,a_n\}$ is said to be \emph{divisibly independent} over $\mathbb{H}$.
For $S \subset \mathbb{D}$, we say that $S$ is \emph{divisibly dependent} over $\mathbb{H}$ if there exist a finite subset $\{a_1,...,a_n\} \subset S$ which is dependent over $\mathbb{H}$. Otherwise, $S$ is said to be \emph{divisibly independent} over $\mathbb{H}$.
Let $\{a_1,...,a_n \} \subset \mathbb{D}$ and let $b \in \mathbb{D}$. We say that $b$ is \emph{divisibly dependent} on $S$ if there exists some $k \in \mathbb{N}$ such that $b^k = \beta \prod_{i=1}^{n}a_{i}^{k_i}$ with $k_i \in \mathbb{Z}$ for $i=1,...,n$ and $\beta \in \mathbb{H}$.
For $S \subset \mathbb{D}$ and $b \in \mathbb{D}$, we say that $b$ is \emph{divisibly dependent} on $S$ if there exists a finite subset  $\{a_1,....,a_n\} \subset S$ such that $b$ is \emph{divisibly dependent} on $\{a_1,...,a_n\}$.
\end{defn}

\begin{rem}\label{newrem1}
Let $m_1,m_2 \in \mathbb{H}[x_1,...,x_n]$ such that $m_2 = \alpha m_1$ where $\alpha \in \mathbb{H}$. If \linebreak $m_1(a_1,...,a_n) = m_2(a_1,...,a_n)$, then we get that $1 = \frac{m_1(a_1,...,a_n)}{m_2(a_1,...,a_n)} = \alpha^{-1}$ in $\mathbb{H}(a_1,...,a_n)$. So $\alpha = 1$, which yields that $m_2 = m_1$.
\end{rem}

\begin{lem}\label{newrem2}
The set $\{a_1, a_2,...,a_n \} \subset \mathbb{D}$ is divisibly dependent over $\mathbb{H}$ if and only if there exists $j \in \{1,...,n\}$ such that $a_j$ is divisibly dependent on $ \{a_1,...,a_{j-1},a_{j+1},...,a_n\}$ over $\mathbb{H}$.
\end{lem}

\begin{proof}
Assume $\{a_1, a_2,...,a_n \}$ is divisibly dependent over $\mathbb{H}$. In view of Remark \ref{newrem1}, w.l.o.g., there exists $j \in \{1,...,n \}$ such that $a_j^{k_j} = m(a_1,...,a_{j-1},a_{j+1},...,a_n)$ where \linebreak $k_j \in \mathbb{N}$ and $m = \beta \prod_{i \in \{1,...,n\} \setminus \{j\}}x_{i}^{s_i}$ where $\beta \in \mathbb{H}$ and $s_i \in \mathbb{Z}$. Thus, by definition, $a_j$ is dependent on $ \{a_1,...,a_{j-1},a_{j+1},...,a_n\}$. Conversely, if $a_j$ is dependent on \linebreak $ \{a_1,...,a_{j-1},a_{j+1},...,a_n\}$ for some $j \in \{1,...,n \}$, then there exists some $k \in \mathbb{N}$ such that $a_j^k = \beta \prod_{i \in \{1,...,n\} \setminus \{ j \} }a_{i}^{k_i}$ with $k_i \in \mathbb{Z}$ and $\beta \in \mathbb{H}$. Multiplying both sides of the equation by
$$\prod_{i : k_i<0} a_{i}^{-k_i} \in \mathbb{H}[a_1,...,a_n],$$ we get that $m_1(a_1,...,a_n) = m_2(a_1,...,a_n)$, where
$$m_1(x_1,...,x_n) = x_j^k \prod_{i : k_i<0} x_{i}^{-k_i}, m_2(x_1,...,x_n) = \beta \prod_{i : k_i>0} x_{i}^{k_i} \in \mathbb{H}[x_1,...,x_n].$$
The monomials $m_1$ and $m_2$ are distinct since $x_j$ appears only in $m_1$.
%In view of the above, the assertion is a direct consequence of divisible dependence being a special case of an abstract dependence relation. A proof can be found in \cite{ComView}.
\end{proof}

\begin{defn}
Let $\mathbb{D}$ be a bipotent domain extending a semifield  $\mathbb{K}$. An element $a \in \mathbb{D}$
is said to be \emph{$\mathbb{K}$-torsion} if $a$ is divisibly dependent on $\mathbb{K}$. We say that $\mathbb{D}$ is $\mathbb{K}$-torsion if every element of $\mathbb{D}$ is $\mathbb{K}$-torsion.
Let $A \subset \mathbb{D}$ be a divisibly independent subset of $\mathbb{D}$ over $\mathbb{H}$, such that all the elements of $A$ are invertible in $\mathbb{D}$. Then the \linebreak $\mathbb{H}$-extension generated by $A$, $\mathbb{H}[A] = \left\{\alpha\prod_{i=1}^{n}a_{i}^{k_i} \ : \ a_i \in A, \ \alpha \in \mathbb{H}, \ k_i \in \mathbb{N} \right\} \subset \mathbb{D}$ is said to be a \emph{divisibly-free}, or \emph{pure transcendental}, extension of $\mathbb{H}$ of \emph{rank} $| A |$ (the number of elements in $A$).
\end{defn}

From the observations made in Remark \ref{construction1} and these last definitions, using the notation of Remark \ref{construction1}, we deduce

\begin{prop}\label{prop1}
Let $\mathbb{D}= \mathbb{H}[a_1,...,a_n]$ as in Remark \ref{construction1}. Let $b_1,...,b_t \in \mathbb{D}$ and $c_{t+1},...,c_{m} \in \mathbb{D}$ be any representatives of the elements $[b_1],...,[b_t] \in G$ and \linebreak $[c_{t+1}], ..., [c_{m}] \in G$ given in Remark \ref{construction1}. Then the following assertions hold:

\begin{enumerate}
  \item For any $j =1,...,t$, $b_j$ is divisibly independent of $$\{b_i \ : \ i \in \{ 1,...,t \} \setminus \{ j \} \} \cup \{a_i \ : \ i=t+1,...,m \}$$ over $\mathbb{H}$.
  \item The set $\{ b_1, ... ,b_t \} \subset \mathbb{D}^{\ast}$ is divisibly independent over $\mathbb{H}$.
  \item For any $i = t+1,...,m$ the element $c_i \in \mathbb{D}^{\ast}$ is $\mathbb{H}$-torsion.
\end{enumerate}
\end{prop}

\begin{proof}
For  $j =1,...,t$, $b_j$ is divisibly independent of
$$\{b_i \ : \ i \in \{ 1,...,t \} \setminus \{ j \} \} \cup \{a_i \ : \ i=t+1,...,m \}$$
over $\mathbb{H}$, for otherwise $b_j^{k} = \alpha \prod_{i \neq j}b_i^{k_i} \prod_{i=t+1}^{m}c_i^{r_i}$ for some $k \in \mathbb{N}$, $k_i,r_i \in \mathbb{Z}$ and $\alpha \in \mathbb{H}$. Thus $$[b_j]^k = \prod_{i \neq j}[b_i]\prod_{i=t+1}^{m}[c_i]^{r_i} \in \left\langle \{[b_i] \ : \ i \in \{ 1,...,t \} \setminus \{ j \}\} \ \cup \{[c_i] \ : \ i=t+1,...,m \} \right\rangle$$ which yields by Remark \ref{construction1} that $[b_j]^{k} = [1]$, contradicting the fact that $[b_j]$ is not a torsion element of $G$. In particular, this implies that $\{ b_1, ... ,b_t \}$ is divisibly independent over $\mathbb{H}$.
For the third assertion, for $i = t+1,...,m$ there exists $n_i \geq 2$ such that \linebreak $[c_i^{n_i}]=[c_i]^{n_i} = [1]$ and thus $c_i^{n_i}=\alpha \in \mathbb{H}$, so $c_i$ is $\mathbb{H}$-torsion.
\end{proof}

\begin{prop}\label{newrem3}
Let $\mathbb{D}$ be a bipotent domain extending $\mathbb{H}$.
Let $A \subset \mathbb{D}$ be a divisibly independent subset of $\mathbb{D}$ over $\mathbb{H}$, such that all of the elements of $A$ are invertible in $\mathbb{D}$. Then $a^{-1} \not \in \mathbb{H}[A]$ for any $a \in A$.
%i.e. $a$ is not invertible in $\mathbb{H}[A]$ for any $a \in A$.
\end{prop}
\begin{proof}
Assume that $a^{-1} \in \mathbb{H}[A]$. Then, there exist $a_1,...,a_n \in A \setminus \{ a \}$, $k_1,...,k_n~\in~\mathbb{N}$, $\alpha~\in~\mathbb{H}^{\ast}$ and $k \in \mathbb{N} \cup \{ 0 \}$, such that $a^{-1} = \alpha a^{k} \prod_{i=1}^{n}a_{i}^{k_i}$. Thus $a^{k+1} = \alpha^{-1} \prod_{i=1}^{n}a_{i}^{-k_i}$. Now, $k \geq 0$ so $k+1 \geq 1$, yielding that $a$ is divisibly dependent on $\{ a_1,...,a_n \} \subset A \setminus \{ a \}$ and so, by Lemma \ref{newrem2}, the set $ \{ a, a_1,...,a_n \}$ is divisibly dependent over $\mathbb{H}$, which yields that $A$ is divisibly dependent over $\mathbb{H}$, a contradiction.
\end{proof}

\begin{cor}\label{cor1}
In the setting of Proposition \ref{newrem3}, we have that $$\mathbb{H}[A] \subset \mathbb{H}(A) \subseteq \mathbb{D},$$ where $\mathbb{H}(A)=Frac(\mathbb{H}[A])$ is the semifield of fractions of $\mathbb{H}[A]$.
\end{cor}
\begin{proof}
 As a straightforward consequence of Proposition \ref{newrem3} we have that \\ $\mathbb{H}[A] \subset \mathbb{H}(A)$ is a proper subset. The second inclusion follows the assumption that every element of $A$ is invertible in $\mathbb{D}$.
\end{proof}

\begin{flushleft} We now turn to study bipotent torsion extensions of semifields.\end{flushleft} %We begin by introducing a well known property of general semifields.
%
%\begin{rem}\label{rem_torsion_free_extensions_copy}
%The multiplicative group of every semifield $\mathbb{H}$ is a torsion-free group, i.e., all of its elements that are not equal to $1$ have infinite order.
%\end{rem}
%\begin{proof}
%If $a^{n} = 1$ for $a \in \mathbb{H}$ and $n \in \mathbb{N}$ then
%$$a(a^{n-1} + a^{n-2} + \dots + a + 1) = a^{n} + (a^{n-1} + \dots + a) = 1 + (a^{n-1} + \dots + a) =$$
%$$a^{n-1} + \dots + a + 1$$
%which yields that $a = 1$.
%\end{proof}

\begin{rem}\label{rem_torsion_free}
The multiplicative group of every semifield $\mathbb{H}$ is a torsion-free group, i.e., all of its elements that are not equal to $1$ have infinite order.
\end{rem}
\begin{proof}
If $a^{n} = 1$ for $a \in \mathbb{H}$ and $n \in \mathbb{N}$ then
$$a(a^{n-1} + a^{n-2} + \dots + a + 1) = a^{n} + (a^{n-1} + \dots + a) = 1 + (a^{n-1} + \dots + a) = a^{n-1} + \dots + a + 1$$ which yields that $a = 1$.
\end{proof}

The following example gives some motivation for our subsequent discussion of bipotent extensions, hopefully making our   definitions and observations very much intuitive.

\begin{exmp}
Let $\mathbb{H}$ be a bipotent semifield and let $\mathbb{D}$ be a semifield extending $\mathbb{H}$. Consider the equation $\lambda^k = \alpha \in \mathbb{H}$ for some $\alpha \neq 1$. The natural question to be asked is under what circumstances does this equation have a solution in $\mathbb{D}$?  We can rewrite the last equality as $\alpha^{-1}\lambda^k = 1$. Assume there exists some $\beta \in \mathbb{H}$ such that $\beta^k = \alpha$. In such a case we have that $(\beta^{-1}\lambda)^k = 1$, following Remark \ref{rem_torsion_free}, we get that $(\beta^{-1}\lambda)^k = 1 \Leftrightarrow \beta^{-1}\lambda = 1$, and so $\lambda = \beta$.\\
This observation is very similar to the classical algebra problem of finding roots for polynomials leading to the theory of algebraic extensions of a field. For instance, similarly to the property of an algebraically closed field,  assuming $\mathbb{H}$ to be divisibly closed in the above setting, will yield the existence of a solution in $\mathbb{H}$ for any equation of the form $\lambda^k = \alpha$.
\end{exmp}

Considering $\tilde{\mathbb{D}} = \mathbb{D}/{\mathbb{H}^{\ast}}$ as defined in Remark \ref{rem_quotient_monoid_stracure}, the following result is merely a straightforward consequence of a well-known result concerning the order of the elements of an abelian monoid. Nevertheless, we choose to write it explicitly.\\

\begin{rem}\label{torrem1}
Let $\mathbb{D}$ be a bipotent domain extending $\mathbb{H}$ and let $a \in \mathbb{D}$ be a torsion element, i.e., $a^k \in \mathbb{H}$ for some $k \in \mathbb{N}$. Torsion powers are an ideal of $\mathbb{Z}$ so are principal. Taking $k \in \mathbb{N}$ to be the generator of the ideal, we have that $a^i \not \in \mathbb{H}$ for any $ i \in
\mathbb{N}$ with $ 1 \leq i <k$ and  $$ \mathbb{H}[a] = \{ \alpha \cdot a^{j}
: \alpha \in \mathbb{H}, \ j = 0, \dots , k-1 \}.$$
\end{rem}

\begin{defn}
In the setting of Remark \ref{torrem1}, define the \emph{degree} of $a \in \mathbb{D}$ over $\mathbb{H}$,
$deg_{ \ \mathbb{H}}(a)$, to be the minimal $k \in \mathbb{N}$ such that $a^k \in \mathbb{H}$. If there
is no such $k$, we define $deg_{ \ \mathbb{H}}(a) = \infty$. Define
the \emph{rank} (or \emph{dimension}) of $\mathbb{H}[a]$ over
$\mathbb{H}$, $[ \mathbb{H}[a] : \mathbb{H} ]$  to be $deg_{ \
\mathbb{H}}(a)$.
\end{defn}

\begin{defn}
Let $\mathbb{D}$ be a bipotent domain extending a semifield $\mathbb{H}$. Then for $a, b_1, \dots b_n \in
\mathbb{D}$, we say that $a$ is \emph{linearly independent} of $\{
b_1, \dots, b_n \}$ over $\mathbb{H}$ if $a \neq \alpha b_i$ for any
$\alpha \in \mathbb{H}$ and any $b_i$. Otherwise, if such $\alpha$
and $b_i$ exist we say that $a$ is \emph{linearly dependent} on $\{ b_1,
\dots, b_n \}$ over $\mathbb{H}$. For any set $\mathrm{B} \subset
\mathbb{D}$ and $a \in \mathbb{D}$, $a$ is said to be  \emph{linearly dependent} on
$\mathrm{B}$ over $\mathbb{H}$ if there exists an element in
$\mathrm{B}$ on which $a$ is dependent over $\mathbb{H}$; Otherwise
$a$ is linearly independent on $\mathrm{B}$ over $\mathbb{H}$. A set
$\mathrm{B} \subset \mathbb{D}$ is \emph{linearly independent} over $\mathbb{H}$
if every $b \in \mathrm{B}$ is linearly independent of $\mathrm{B} \setminus
\{ b \}$. We say that $\mathrm{B} \subset \mathbb{D}$ \emph{spans}
$\mathbb{D}$ over $\mathbb{H}$ if any $a \in \mathbb{D}$ is
linearly dependent on $\mathrm{B}$ over $\mathbb{H}$. If $\mathrm{B}$ is
linearly independent over $\mathbb{H}$ and \emph{spans} $\mathbb{D}$ over
$\mathbb{H}$, i.e. $ \mathbb{D} = \{ \alpha  w : \alpha \in
\mathbb{H}, \ w \in \mathrm{B} \}$, we say that $\mathrm{B}$ is a
\emph{basis} of $\mathbb{D}$ over $\mathbb{H}$. We define $[
\mathbb{D} : \mathbb{H} ]$ to be $|B| \in \mathbb{N} \cup \{\infty
\}$.
\end{defn}

\begin{note}
In what follows, in order to avoid confusion with classical algebra, we refer to `linear dependence' defined above as `dependence'.
\end{note}

\begin{rem}

\begin{enumerate}
  \item Although the notion of dependence is essentially binary, it
corresponds to linear dependence in classical algebra. Consider the expression
$\sum_{i=1}^{n}\alpha_i b_i$ where $\{b_1,...,b_n \}$ is a set of  independent
elements. By the definition of independence, we have that $\alpha_i b_i \neq \alpha_j b_j$ for
every $1 \leq i < j \leq n$. Thus $\sum_{i=1}^{n}\alpha_i b_i = \alpha_k b_k$
for some $\alpha_k b_k$, $1 \leq k \leq n$,  by bipotency. So $a = \sum_{i=1}^{n}\alpha_i b_i$ implies that
$a = \alpha_k b_k$ for the appropriate index k.

\item In terms of the above definitions, Remark \ref{torrem1} shows that $$\{ \alpha \cdot a^j :
\alpha \in \mathbb{H}, \  0 \leq j \leq deg_{\mathbb{H}}(a)-1 \}$$ is a basis of
$\mathbb{H}[a]$ over $\mathbb{H}$.

  \item Let $A \subset \mathbb{D}$ be a set such that for any $a,b \in A$, $a$ and $b$ are independent over $\mathbb{H}$. Then $A$ is independent over $\mathbb{H}$.  This is a straightforward consequence of the abstract dependence relation.

\item  If $a$ is not $\mathbb{H}$-torsion then $a^m \neq \alpha a^{k}$ for any $\alpha \in \mathbb{H}$ and any $m > k \geq 0$ since otherwise $a^{m-k} = \alpha \in \mathbb{H}$, contradicting the assumption that $a$ is not $\mathbb{H}$-torsion. Thus, the set $\{ a^{k} : k \geq 0 \}$ is independent over $\mathbb{H}$, and so
    $[ \mathbb{H}[a] : \mathbb{H} ] = \infty$.

\end{enumerate}
\end{rem}

\ \\

Lemma \ref{lem1}, Remark \ref{rem_transitivity_of_torsion} and Proposition \ref{torprop1} which we now introduce, are all straightforward consequences of well known results in the theory of abelian monoids,
when considering $\tilde{\mathbb{D}} = \mathbb{D}/{\mathbb{H}^{\ast}}$ as defined in Remark \ref{rem_quotient_monoid_stracure}, where $\mathbb{D}$ is a bipotent domain extending the semifield $\mathbb{H}$.

\begin{lem}\label{lem1}
Let $\mathbb{H}$ be a semifield and let $\mathbb{D}, \mathbb{K}$ be bipotent domains
extending $\mathbb{H}$ such that $\mathbb{H} \subset \mathbb{K} \subset \mathbb{D}$.
Then $[ \mathbb{D} : \mathbb{H} ] \geq [ \mathbb{K} : \mathbb{H} ]$.
Moreover, if $\mathbb{K}$ is a semifield then  $[ \mathbb{D} : \mathbb{H} ] = [ \mathbb{D} :
\mathbb{K} ] \cdot [ \mathbb{K} : \mathbb{H} ].$
\end{lem}

\begin{proof}
In order to simplify notations of the proof, $w_0=u_0=1_{\mathbb{H}}$ denotes the identity element of $\mathbb{H}$ with respect to multiplication inside $\mathbb{D}$ and $\mathbb{K}$, respectively. \newline Let $[ \ \mathbb{D} : \mathbb{H}
\ ] = s < \infty$, thus $ \mathbb{D} = \{ \alpha w_j : \alpha \in
\mathbb{H}, \ w_j \in \mathbb{D}, \ j = 0, \dots , s-1 \}$ for a base $\{w_0,...,w_{s-1} \}$. If $[ \
\mathbb{K} : \mathbb{H} \ ] = t > s$, then writing
$$ \mathbb{K} = \{ \beta  \cdot u_i : \beta \in \mathbb{H}, \ u_i \in \mathbb{K},  \ i = 0, \dots , t-~1~\}$$
for a base $\{u_0,...,u_{t-1} \}$, as $\mathbb{K} \subset \mathbb{D}$ we have that for
each $i = 0, \dots , t-1$ , $u_i = \alpha_{i,j} w_j$ for some $j \in
\{ 0, \dots , s-1 \} $. Since $t>s$ there exist $i,j \in \{0, \dots
, t-1 \}$,  $ i < j$, and  $\alpha_1 , \alpha_2 \in \mathbb{H} $
such that $ u_i = \alpha_1 w_k$ and $ u_j = \alpha_2 w_k$ for some $k \in
\{0, ... , s-1 \}$. Thus, multiplying the last two equations by
$\alpha_2$ and $\alpha_1$ respectively, we get that $u_i =  \gamma
u_j$ where $\gamma = \frac{\alpha_1 \beta_2}{\alpha_2 \beta_1} \in
\mathbb{H}$ contradicting the fact that $u_i$ and $u_j$ are
independent over $\mathbb{H}$. Thus $t \leq s$ and we have proved the first assertion.
For the second assertion, the only nontrivial case is when
$\mathbb{D}$ is of finite rank over $\mathbb{K}$ and $\mathbb{K}$ is
of finite rank over $\mathbb{H}$, say $s$ and $t$, respectively. In such a
case, using the above notation, \small$$\mathbb{D} =
\{ a \cdot w_j  :  a \in \mathbb{K}, \ j = 0, \dots, s-1 \} = \{ \alpha \cdot u_r \cdot w_j  :  \alpha \in \mathbb{H}, \ r = 0, \dots, t-1, \ j = 0, \dots, s-1 \}.$$ \small
We argue that the $u_r w_j$ are independent
over $\mathbb{H}$. Indeed, if $\{ u_{r}w_{j} \}$ are dependent over $\mathbb{H}$, then there exist $u_{r_1} w_{j_1}$ and  $u_{r_2} w_{j_2}$ such that $\alpha_1u_{r_1}w_{j_1} = \alpha_2 u_{r_2} w_{j_2}$ where $u_{r_1}
\neq u_{r_2}$ or $w_{j_1} \neq w_{j_2}$ and  $\alpha_1, \alpha_2 \in \mathbb{H}$ non-zero. Then $w_{j_1} = \left(\frac{\alpha_2u_{r_2}}{\alpha_1u_{r_1}}\right) w_{j_2} = a w_{j_2}$ with $a \in \mathbb{K}$. So,
$w_{j_1},w_{j_2}$ are dependent over $\mathbb{K}$, contradicting our assumptions. Thus by definition, we conclude that $$\left[ \mathbb{D} : \mathbb{H} \right] = \left| \{u_r \cdot w_j \ :  \ r = 0, \dots, t-1, \ j = 0,
\dots s-1 \}\right| = s \cdot t = \left[ \mathbb{D} : \mathbb{K} \right] \cdot \left[ \mathbb{K} : \mathbb{H} \right],$$ as desired.
\end{proof}

\begin{rem}
By Lemma \ref{lem1}, it is evident that if $\mathbb{H} \subseteq
\mathbb{H'}$ are two bipotent semifields, then
$deg_{ \ \mathbb{H'}}(a) \leq  deg_{ \ \mathbb{H}}(a).$
\end{rem}

\begin{rem}\label{torrem2}
Let $\mathbb{H}$ be a semifield and let $\mathbb{D}$ be a bipotent domain extending $\mathbb{H}$.
Let $\{a_1, a_2,...,a_n \} \subset \mathbb{D}$ such that $a_i$ is $\mathbb{H}$-torsion for $i=1,...,n$. If $b \in \mathbb{D}$ such that $b$ is dependent on $\{a_1, a_2,...,a_n \}$ over $\mathbb{H}$, then $b$ is $\mathbb{H}$-torsion.
\end{rem}

\begin{rem}\label{rem_transitivity_of_torsion}
If $\mathbb{D}_1 \subset \mathbb{D}_2$ are bipotent extensions of a semi-field $\mathbb{H}$, such that
$\mathbb{D}_1$  is $\mathbb{H}$-torsion and $\mathbb{D}_2$ is
$\mathbb{D}_1$-torsion, then $\mathbb{D}_2$ is $\mathbb{H}$-torsion. Indeed, as  $\mathbb{D}_1$ is $\mathbb{H}$-torsion, for any  $ \alpha \in
\mathbb{H}$ there exist $k \in \mathbb{N}$ and $a \in \mathbb{D}_1$
such that $\alpha = a^k$. Since $\mathbb{D}_2$ is $\mathbb{D}_2$-torsion ,
there exists $m \in \mathbb{N}$ and $b \in
\mathbb{D}_1$ such that $a = b^m$. So $\alpha = a^k = b^{mk}$. Since this is
true for any  $ \alpha \in \mathbb{H}$ we have $\mathbb{D}_2$ is $\mathbb{H}$-torsion as desired.
\end{rem}

\begin{prop}\label{torprop1}
Let $\mathbb{H}$ and $\mathbb{D}$ be as above. Let $\mathbb{H}[a]
\subset \mathbb{D}$ be the bipotent domain extending $\mathbb{H}$ generated
by $a \in \mathbb{D}$ over $\mathbb{H}$. Then $\mathbb{H}[a]$ is a bipotent semifield iff
$a$ is an $\mathbb{H}$-torsion element.
\end{prop}

\begin{proof}
Let $a \in \mathbb{D}$ be $\mathbb{H}$-torsion. If $a \in \mathbb{H}$ then $\mathbb{H}[a] =
\mathbb{H}$ and we are done. Therefore we may assume that $a \in
\mathbb{D} \setminus \mathbb{H}$. Let $k \in \mathbb{N}$ be the
minimal natural number such that $a^k = \beta$ for some $\beta \in
\mathbb{H}$. In view of the remark above, such $k$ exists. Let $b \in
\mathbb{H}[a]$, then by the construction above we can write $b = f(a)$
where $ f(x) = \alpha x^m \in \mathbb{H}[x]$ is a monomial, and $m \in
\mathbb{N} \cup \{0\}$, i.e., $b = \alpha a^m$. Taking $m = qk+r$ where $r,q \in \mathbb{N}
\cup \{0\}$ such that $r<k$ or $r=0$, we can rewrite $b = \alpha a^m = \alpha a^{kq} a^{r}=  (\alpha \beta^q) a^{r}$. Notice that if $r=0$ we get $b \in \mathbb{H}$, thus invertible. So we may assume $r>0$ and $(k-r) > 0$.
Let $c=(k-r)a$. Then  $bc =  (\alpha \beta^{q})a^k  = \alpha \beta^{q+1} \in
\mathbb{H}$, and thus $bc$ is invertible. Taking
$\gamma = (bc)^{-1}$ and defining $c' = \gamma c$, we get $bc' =
c'b=1$, so $b$ is invertible and $\mathbb{H}[a]$ is a semi-field. \\
For the second direction of the assertion, assume $\mathbb{H}[a]$ is a
semi-field extending $\mathbb{H}$. Since $a \in \mathbb{H}[a]$ we
have that $a$ is invertible in $\mathbb{H}[a]$ thus we can write
$a^{-1} = (\delta)a^{t}$ for some $t \in \mathbb{N} \cup
\{0\}$, $\delta \in \mathbb{H}$. From this we have
that $a^{t+1} = \delta$. Thus, by definition, $a$ is an $\mathbb{H}$-torsion element.
\end{proof}

\begin{cor}\label{torcor1}
By Proposition \ref{torprop1}, it is evident that $\mathbb{H}[a]$ is a bipotent semifield if and only if $a^{-1} \in
\mathbb{H}[a]$. We deduce that $\mathbb{H}[a]$ is not a bipotent semifield
if and only if  $a^{-k} \not \in \mathbb{H}[a]$ for any $k \in \mathbb{N}$.
Using Proposition \ref{torprop1}, we have that if $\mathbb{H}[a]$ is not a bipotent semifield
then $a$ is not $\mathbb{H}$-torsion, which in turn implies that  $a^{-i} \neq a^{-j}$ are independent over $\mathbb{H}$ for any $i,j \in \mathbb{N}$ such that  $i < j$ (for otherwise $a^{j-i}= \beta \in \mathbb{H}$). Consequently, we get that $[ \ \mathbb{H}(a) : \mathbb{H}[a] \ ] = deg_{\mathbb{H}[a]}(a^{-1}) = \infty.$
\end{cor}

\begin{proof}
Since $b$ is dependent on $\{a_1, a_2,...,a_n \}$ over $\mathbb{H}$ we have that $b^k = \beta \prod_{i=1}^{n}a_{i}^{k_i}$ with $\beta \in \mathbb{H}$, $k \in \mathbb{N}$ and $k_i \in \mathbb{Z}$.
For $i = 1,...,n$, $a_{i}$ is $\mathbb{H}$-torsion, thus so is $c_{i} = a_{i}^{k_i}$, so there exists $s_i \in \mathbb{N}$, minimal, such that $c_{i}^{s_i} \in \mathbb{H}$. Let $s = lcm(s_1,...,s_n)$, then there are $r_i \in \mathbb{N}$ such that $s = r_is_i$ and $b^{ks} = (b^k)^s = \prod_{i=1}^{n}(c_i^{s_i})^{r_i} = \prod_{i=1}^{n}(\prod_j b_j^{k_{i,j}})^{r_i} \in \mathbb{H}$. Thus $b$ is $\mathbb{H}$-torsion.
\end{proof}

\begin{cor}\label{torcor2}
By Remark \ref{torrem2}, the $\mathbb{H}$-torsion elements in $\mathbb{D}$ form a sub-domain $Tor_{\mathbb{H}}(\mathbb{D})$ of $\mathbb{D}$. Moreover, if $\mathbb{D}$ is a semifield, so is $Tor_{\mathbb{H}}(\mathbb{D})$. Indeed, let $a,b \in \mathbb{D}$ be $\mathbb{H}$-torsion. Since $ab$ and $a+b \in \{ a, b \}$, are divisibly dependent on $\{ a,b \}$, by Remark \ref{torrem2} they are $\mathbb{H}$-torsion.
\end{cor}

\begin{cor}
Let $\mathbb{D}=\mathbb{H}[a_1,...,a_n]$ be the affine bipotent domain extending $\mathbb{H}$ generated by $a_1,...,a_n \in \mathbb{D}$ , then $\mathbb{D}$ is a $\mathbb{H}$-torsion iff $a_1,...,a_n$ are $\mathbb{H}$-torsion elements.
\end{cor}
\begin{proof}
Since $a_1,...,a_n$ generate $D$, by Corollary \ref{torcor2}, $\mathbb{D}$ is $\mathbb{H}$-torsion. The inverse direction follows the definition of an $\mathbb{H}$-torsion bipotent domain.
\end{proof}

\begin{prop}\label{torprop2}
Let $\mathbb{D} = \mathbb{H}[a_1,...,a_n]$ be an affine bipotent domain extending a semifield $\mathbb{H}$.
If $\mathbb{D}$ is $\mathbb{H}$-torsion then $\mathbb{D}$ is a semifield.
\end{prop}
\begin{proof}
Assume $\mathbb{D}$ is $\mathbb{H}$-torsion, then by
definition, each $a_i$ is an $\mathbb{H}$-torsion element. Define
inductively $\mathbb{H}_i = \mathbb{H}[a_1, \dots a_i] =
\mathbb{H}_{i-1}[a_i]$ which is a semifield for each $i$ by
Proposition \ref{torprop1}, since $\mathbb{H}_{i-1}$ is a semifield by
induction (Indeed, $a_i$ is an $\mathbb{H}_{i-1}$-torsion since $\mathbb{H} \subset \mathbb{H}_{i-1}$).
In Particular, $\mathbb{D} = \mathbb{H}_n$ is a semi-field.
Furthermore, if $d_i = deg_{\mathbb{H}}(a_i)$ then $[\mathbb{H}_i \ : \ \mathbb{H}_{i-1}] =
deg_{\mathbb{H}_{i-1}}(a_i) \leq deg_{\mathbb{H}}(a_i) = d_i$
implying that $$ [ \mathbb{D} : \mathbb{H} ] =[ \mathbb{H}_n :
\mathbb{H}_{n-1} ][ \mathbb{H}_{n-1} : \mathbb{H}_{n-2}] \dots [
\mathbb{H}_1 : \mathbb{H}] \leq \prod_{i=1}^{n}d_i.$$
\end{proof}

\begin{thm}\label{mainthm}
Let $\mathbb{D}= \mathbb{H}[a_1,...,a_n]$ be an affine bipotent semifield extending $\mathbb{H}$.
Then for suitable elements $b_1,...,b_t,c_{t+1},...,c_m \in \mathbb{D}^{\ast}$,
\begin{equation}\label{maineq}
\mathbb{D} = \left(\mathbb{H}(b_1,...,b_t)\right)[c_{t+1},...,c_{m}]=\left(\mathbb{H}[c_{t+1},...,c_{m}]\right)(b_1,...,b_t)
\end{equation}
(and so, $\mathbb{H}(b_1,...,b_t)$ is a divisibly-free (pure transcendental) $\mathbb{H}$-semifield extension), and $\mathbb{H}[c_{t+1},...,c_{m}]$ is an $\mathbb{H}$-torsion semifield extension.
\end{thm}

\begin{proof}
First, since $b_1,...,b_t \in \mathbb{D}^{\ast}$, $b_1,...,b_t \in \mathbb{D}$ are invertible in $\mathbb{D}$. By Corollary \ref{cor1}, the semifield of fractions $\mathbb{K} = \mathbb{H}(b_1,...,b_t) \subset \mathbb{D}$. Now, $c_{t+1},...,c_m$ are divisibly dependent on $\mathbb{H}$, thus $\mathbb{H}[c_{t+1},...,c_m]$ is $\mathbb{H}$-torsion, and since $\mathbb{H} \subset \mathbb{K}$, $\mathbb{K}[c_{t+1},...,c_m] \subset \mathbb{D}$ is $\mathbb{K}$-torsion and thus, by Proposition \ref{torprop2}, a semifield over $\mathbb{K}$. Since the generating elements, $a_1,...,a_n$, of $\mathbb{D}$ over $\mathbb{H}$ are in $\mathbb{H}[b_1,...,b_t,c_{t+1},...,c_m] \subset \mathbb{K}[c_{t+1},...,c_m]$, it follows that $\mathbb{D} \subset \mathbb{K}[c_{t+1},...,c_m]$. This proves the left equality of equation \eqref{maineq}. For the second equality, let $\mathbb{E} = \mathbb{H}[c_{t+1},...,c_{m}]$. Since $\mathbb{E}$ is $\mathbb{H}$-torsion it is a semifield by Proposition \ref{torprop2}. By Proposition \ref{prop1}, $b_1,...,b_t$ are divisibly independent of $\{a_i  :  i~=~t~+~1~,...,m \}$ over $\mathbb{H}$. Thus $\{ b_1, ... ,b_t \}$ is divisibly independent over $\mathbb{E}$, implying  that $\mathbb{E}(b_1,...,b_t) \subset \mathbb{D}$ is divisibly-free over $\mathbb{E}$. The same argument used in the first part of the proof yields that $\mathbb{E}(b_1,...,b_t)~=~\mathbb{D}$.
\end{proof}

\begin{rem}\label{td_rem}
Using the notation of the theorem, the number of divisibly independent elements, $t$, does not depend on the
choice of these elements. Indeed, these are basis elements generating the free part of the direct sum decomposition of the finitely generated abelian group .
Thus their number is invariant under any choice of basis. \\

Note also that by Proposition \ref{torprop2} we have that $$[\mathbb{H}[c_{t+1},...,c_m] : \mathbb{H}] = [\mathbb{K}[c_{t+1},...,c_m] : \mathbb{K}]\leq \prod_{i=t+1}^{m} deg_{\mathbb{H}}(c_i).$$
Again, recalling the fundamental theorem on abelian groups, the number \linebreak $\prod_{i=t+1}^{m} deg_{\mathbb{H}}(c_i)$ does not depend of our choice of torsion elements $ \{ c_i \}$.
\end{rem}
\bigbreak

%\begin{flushleft} Theorem \ref{mainthm} and Remark \ref{td_rem} show that the following is well defined:\end{flushleft}
%\begin{defn}
%Let $\mathbb{D}= \mathbb{H}[a_1,...,a_n]$ be an affine bipotent semifield extending $\mathbb{H}$.
%The \emph{transcendence degree} of $\mathbb{D}$ over $\mathbb{H}$ is defined to be the unique transcendence degree
%of the divisibly free semifield $\mathbb{H}(b_1,...,b_t) \subset \mathbb{D}$ extending $\mathbb{H}$.
%\end{defn}
%

\ \\
\section{Extensions of cancellative semifields}\label{section:CancellativeExtensions}\ \\

\begin{note}
Let $\mathbb{H}$ be a semifield, and let $\mathbb{H}[x]$ be its corresponding polynomial semiring. In this section, when $\mathbb{H}[x]$ is taken as a denominator of a quotient semiring, $\mathbb{H}[x]$ is assumed not to include the zero polynomial.
\end{note}

\begin{defn}
Let $(\mathbb{H}, +, \cdot)$ be a semifield. $\mathbb{H}$ is said to be \emph{cancellative} if
\begin{equation*}
\forall  \ a,b,c \in \mathbb{H} \ \ a + c = b + c \Rightarrow a = b.
\end{equation*}
\end{defn}

\begin{note}
In the rest of this section, a cancellative semifield means cancellative with respect to addition. In case we refer to cancellation with respect to multiplication, we indicate it explicitly.
\end{note}

\begin{note}
The definition of affine extensions of a semifield introduced in section~\ref{section:SemifieldsPreliminaries}, also applies to cancellative semifields. Thus we omit stating them in this section.
\end{note}

%\begin{defn}
%A semifield $\mathbb{S}$ is said to be \emph{proper} semifield if it is not a field.
%\end{defn}

Let $(\mathbb{H}, + ,\cdot)$ be a semiring. Then $\mathbb{H}$ is embeddable into a ring iff $(\mathbb{H},+)$ is commutative cancellative (\cite[Theorem (5.1)]{AlgTheAndApp}).
In particular, any commutative and cancellative semiring $\mathbb{H}$ has a well-known extension to a ring, called its `difference ring' to be \linebreak defined next.

\begin{defn}
Let $(\mathbb{H}, + ,\cdot)$ be a commutative cancellative semiring. Let $R$ be the ring consisting of all differences $a-b$ with $a,b \in \mathbb{H}$, subjected to some elementary rules of identification (cf. \cite[Chapter 2]{AlgTheAndApp}). This ring $R$ is uniquely determined (up to isomorphisms leaving $\mathbb{H}$ fixed) and is the minimal ring containing $\mathbb{H}$. $R$ is called the \emph{difference ring} or the \emph{ring of differences} of $\mathbb{S}$ and is denoted by $R = D(\mathbb{S})$.
\end{defn}

\begin{note}
From now on, we always assume a semiring to be commutative and \linebreak (additively) cancellative. Moreover, we require a semifield to consist of at least two elements.
\end{note}

\begin{lem}\cite{Hom_Semifields}
Let $\mathbb{S}$ and $\mathbb{A}$ be proper semifields and let $\Phi: \mathbb{S} \rightarrow \mathbb{A}$ be a semifield homomorphism.
Then $\Phi$ determines a congruence $\chi= \Phi^{-1} \circ \Phi$ on $\mathbb{S}$. The class $[1]_\chi$ corresponding to the
(multiplicative) identity element of  \ $\mathbb{S}$ is a kernel.
\end{lem}
\ \\
\begin{prop}\label{cor6_2}\cite[Corollary (6.2)]{Hom_Semifields}
Let $\mathbb{S}$ be a cancellative proper semifield and let $R = D(\mathbb{S})$ be its difference ring. Then each ideal $A$ of $R$ defines a kernel $$\hat{K}~=~( 1 + A ) \cap \mathbb{H}$$ of \ $\mathbb{S}$ such that $\mathbb{S}/{\hat{K}}$ is cancellative, and each kernel  $K$ of \ $\mathbb{S}$ defines an ideal $$\hat{A} = (K - 1)\mathbb{S}$$ of $R$. Moreover $A \rightarrow \hat{K}$ defines an isomorphism of the lattice of all ideals of $R$ onto the lattice of all the kernels $K$ of \ $\mathbb{S}$ such that $\mathbb{S}/{K}$ is cancellative.\\
In particular, $R=D(\mathbb{S})$ is a simple ring (hence a field, in the case where $(\mathbb{S}, \cdot)$ is commutative) if and only if $\mathbb{S}$ has no kernels $K$ except $K = \{ 1 \}$ and $K = \mathbb{S}$.
\end{prop}

\begin{defn}
Let $\mathbb{S}$ be  a commutative cancellative semifield extending  a cancellative semifield $\mathbb{H}$. An element $a \in \mathbb{S}$ is said to be \emph{algebraic} over $\mathbb{H}$ if there exists a non-zero polynomial $g(x) \in D(\mathbb{H})[x]$ such that $g(a) = 0$. Otherwise $a \in \mathbb{S}$ is said to be \emph{transcendental} over $\mathbb{H}$. Finally, $\mathbb{S}$ is said to be \emph{algebraic} over $\mathbb{H}$ if each element $a \in \mathbb{S}$ is algebraic over $\mathbb{H}$.
\end{defn}

\begin{lem}\label{field_lemma}\cite[Lemma (6.7)]{Hom_Semifields}
Let $\mathbb{S}$ be a semifield extension of a proper semifield $\mathbb{H}$. Then $\mathbb{S}$ is a field if and only if \ $\mathbb{S}$ contains an element $a \neq 0$ which is a zero of a polynomial $h(x) \in \mathbb{H}[x]$.
\end{lem}

\begin{rem}
Note that the definition of algebraic element requires $a \in \mathbb{S}$ to annihilate a polynomial in $D(\mathbb{H})[x]$, not necessarily in $\mathbb{H}[x]$.
\end{rem}

\begin{note}\label{note1}
In view of this lemma, we assume from now on that every algebraic element annihilates a polynomial in $D(\mathbb{H})[x] \setminus \mathbb{H}(x)$, for otherwise, the extension forms a field rather than a proper semifield, and thus has no interest for our purposes.
\end{note}

\begin{thm}\cite[Satz 1]{Satz}
Let $\mathbb{S}$ be an algebraic semifield extension of \ $\mathbb{H}$ and assume that $D(\mathbb{H})$ is a field. Then $\mathbb{S}$ is embeddable into a field if and only if $D(\mathbb{S})$ is a field.
\end{thm}

\begin{rem}
Let $\mathbb{H}$ be a commutative cancellative semifield and let $\mathbb{S}$ be a commutative cancellative semifield which is a simple semifield extension of $\mathbb{H}$. Then, by Definition \ref{defn_simple_semifield_extension},
there exists an element $d \in \mathbb{S}$ that generates $\mathbb{S}$ as a semifield over $\mathbb{H}$, i.e. $$\mathbb{S} = \left\{ \frac{f(d)}{g(d)} \ : \ f,g \in \mathbb{H}[x], \ g \neq 0 \right\}.$$
Note that by our previous assumption on algebraic elements stated in Note \ref{note1}, we have that $g(d)=0 \Leftrightarrow g=0$ for $g \in \mathbb{H}[x]$.
\end{rem}

\begin{lem}\cite[Lemma (7.1)]{Hom_Semifields}
For each commutative cancellative proper semifield $\mathbb{H}$ there exists, up to isomorphism, a unique simple proper semifield extension
$$\mathbb{S}= \mathbb{H}(x)= \{\frac{f(x)}{g(x)} \ : \ f,g \in \mathbb{H}[x], \ g \neq 0 \}$$
such that $x$ is transcendental over $D(\mathbb{H})$. Here $\mathbb{H}(x)$ is the  semifield of quotients of the polynomial semiring over $\mathbb{H}$ in $x$.
\end{lem}

\begin{prop}\label{prop_cancelative_ext_stracture}\cite[Chapter 9]{Cancellative1}
\begin{enumerate}
  \item The difference ring of $\mathbb{H}(x)$ is $$D(\mathbb{H}(x))= \frac{D(\mathbb{H})[x]}{\mathbb{H}[x]} = D(\mathbb{H})[x]_{\mathbb{H}[x]}$$ where $\frac{D(\mathbb{H})[x]}{\mathbb{H}[x]}$ is the set of all quotients of the form $\frac{g(x)}{h(x)}$ with $g(x) \in D(\mathbb{H})[x]$ and $0 \neq h(x) \in \mathbb{H}[x]$ while $D(\mathbb{H})[x]_{\mathbb{H}[x]}$ is the localization of the ring $D(\mathbb{H})[x]$ by the
      multiplicative set $\mathbb{H}[x]$. %One can view $D(\mathbb{H}(x))$ as an additive localization by the set $(D(\mathbb{H})[x] \setminus \mathbb{H}[x])/\mathbb{H}[x] = D(\mathbb{H}(x))\setminus \mathbb{H}(x)$.
  \item $\mathbb{H}(x)$ is embeddable into a field iff $D(\mathbb{H}(x))$ has no zero divisors which holds iff $D(\mathbb{H})$ has no zero divisors, which in turn holds iff \ $\mathbb{H}$ is embeddable into a field.
  \item $\mathbb{H}(x)$ is universal for all simple proper semifield extensions of $\mathbb{H}$ in the sense that for any $d \in \mathbb{S}$ extending $\mathbb{H}$, $\mathbb{H}(d) \cong \mathbb{H}(x)/{K}$ for at least one kernel $K$ of $\mathbb{H}(x)$ such that $K \cap \mathbb{H} = \{ 1 \}$.
      The last semiring isomorphism is defined by the map $\psi : \mathbb{H}(x)/{K} \rightarrow \mathbb{H}(d)$ sending $xK \mapsto d$ and fixing $\mathbb{H}$.
\end{enumerate}
\end{prop}

\ \\

\begin{note}
In the notation of Proposition \ref{prop_cancelative_ext_stracture}, we have that $\mathbb{H}[x] \subset \mathbb{H}(x)$ and \linebreak $\mathbb{H}[x] \subset D(\mathbb{H})[x]=D(\mathbb{H}[x])$, moreover
$\mathbb{H}(x) \subset D(\mathbb{H}(x))$ and $D(\mathbb{H})[x] \subset D(\mathbb{H}(x))$ the last inclusion is obtained via the localization map of $D(\mathbb{H})[x]$ with respect to the multiplicative subset $\mathbb{H}[x]$.
\end{note}

\ \\

\begin{prop}\label{algebraic_prop}
Let $\mathbb{S} = \mathbb{H}(d)$ be a proper semifield extension of the semifield $\mathbb{H}$, where $d \in \mathbb{S}$. If $D(\mathbb{S})$ is a field, then $d$ is algebraic over $\mathbb{H}$.
\end{prop}
\begin{proof}
Assume $d \in \mathbb{S}$ is not algebraic. Thus, we may replace $d$ by an indeterminate $x$. Now, $D(\mathbb{S}) = D(\mathbb{H}(x))= \frac{D(\mathbb{H})[x]}{\mathbb{H}[x]}$ is a field, thus, is simple. The latter yields that for any ideal $A \lhd D(\mathbb{H})[x]$,  $A \cap \mathbb{H}[x] \neq \emptyset$.
Consider the ideal $B = p(x)D(\mathbb{H})[x] \lhd D(\mathbb{H})[x]$ generated by $p(x) = x - \alpha$ with $0 \neq \alpha \in \mathbb{H}$, then $\alpha$ is a root for any polynomial $f(x) \in B$. If there exists $0 \neq f(x) \in \mathbb{H}[x] \cap B$, then $f(\alpha) = 0$ and so Lemma \ref{field_lemma} implies that $\mathbb{H}$ is a field, which contradicts our assumption of $\mathbb{H}$ being a proper semifield. We thus conclude that $D(\mathbb{H})[x]$ has no proper ideals and thus is a field, which yields that $x$ is algebraic over $\mathbb{H}$ (i.e. algebraic over $D(\mathbb{H})$), contradiction.
\end{proof}

\begin{cor}
As our interest lies in affine cancellative semifields extensions $\mathbb{S}$  of \ $\mathbb{H}$ such that $D(\mathbb{S})$ is a field, applying Proposition \ref{algebraic_prop} inductively yields that $\mathbb{S}$ must be an algebraic extension of $\mathbb{H}$.
\end{cor}

\begin{thm}\label{thm1}\cite[Theorem (7.2)]{Hom_Semifields}
Let $\mathbb{H}$ be a commutative cancellative semifield and $\mathbb{S}= \mathbb{H}(x)$ a simple proper semifield extension of $\mathbb{H}$ where $x$ is transcendental over $D(\mathbb{H})$. Then the lattice of ideals $A$ of $D(\mathbb{S})$  is isomorphic to the lattice of kernels $K$ of $\mathbb{S}$ such that $\mathbb{S}/K$ is cancellative by $$A \mapsto K=(1+A) \cap \mathbb{S}.$$
The lattice of ideals $A$ of $D(\mathbb{S})$ is isomorphic to the lattice of ideals $A'$ of $D(\mathbb{H})[x]$ by the well-known map $$A \mapsto A'= A \cap D(\mathbb{H})[x]$$
with the property that for any $h(x) \in \mathbb{H}[x]$ and any $g(x) \in D(\mathbb{H})[x]$:
\begin{equation}\label{idealeq}
h(x)g(x) \in A' \Rightarrow g(x) \in A'.
\end{equation}
The above isomorphisms comprise a bijective correspondence between the lattice of the specified kernels and the lattice of ideals in $D(\mathbb{H})[x]$ given by
$$\forall a(x),b(x) \in \mathbb{H}[x] \ : \ \frac{a(x)}{b(x)} \in K  \Leftrightarrow a(x) - b(x) \in A'.$$
\end{thm}

\ \\

\begin{prop}\label{prop_cancellative_supplement}\cite[Supplement (7.3)]{Hom_Semifields}
Under the notation of Theorem \ref{thm1}, the \linebreak following assertions hold:\\
\textbf{1.} \  Excluding the trivial case where $A' = D(\mathbb{H})[x]$ corresponding to $K = \mathbb{S}$, we have that $A' \cap \mathbb{H}[x] = \emptyset$, and $A'\cap D(\mathbb{H}) = \{ 0 \}$ is equivalent to $K \cap \mathbb{H} = \{ 1 \}$. If the latter holds, the natural epimorphism $\Psi: \mathbb{S} \rightarrow \mathbb{S}/K$ induces an isomorphism on $\mathbb{H}$. We consider the proper cancellative semifield $\mathbb{S}/K$ as a simple semifield extension $\mathbb{H}(\bar{x})$ of $\mathbb{H}$ for $\bar{x} = xK = \Psi(x)$.
In this setting (see Remark \ref{rem2}), we have that
\begin{itemize}
  \item $\mathbb{H}[\bar{x}] \subset \mathbb{S}/K = \mathbb{H}(\bar{x}) \subset D(\mathbb{S}/K)$.
  \item $\mathbb{H}[\bar{x}] \subset D(\mathbb{H})[x]/A' \cong D(\mathbb{H})[\bar{x}] = D(\mathbb{H}[\bar{x}]) \subset R/A$ with $R = \frac{D(\mathbb{H})[x]}{\mathbb{H}[x]}$.
  \item $R/A \cong D(\mathbb{S}/K)$.
\end{itemize}
In particular, $A'$ consists of all polynomials $g(x) \in D(\mathbb{H})[x]$ \ such that \ $g(\bar{x}) = 0$. \\
Clearly, $A' = \{ 0 \}, K = \{ 1 \}$ and $\bar{x}$ is transcendental over $D(\mathbb{H})$ are mutually equivalent.\\
\textbf{2.} \ Let $\mathbb{H}$ be a commutative cancellative semifield and let $\mathbb{H}(d)$ be a proper semifield extension of  $\mathbb{H}$. Then $$A' = \{ g(x) \in D(\mathbb{H})[x] \ : \ g(d) = 0 \}$$ is an ideal of $D(\mathbb{H})[x]$ which satisfies \eqref{idealeq}, and $\mathbb{H}(x)/K = \mathbb{H}(\bar{x}) \cong \mathbb{H}(d)$ holds for the corresponding kernel $K$ of $\mathbb{H}(x)$.
\end{prop}

\ \\

\begin{rem}\label{rem2}
Notice that $K \neq \mathbb{S}$ implies that $K \cap \mathbb{H} \neq \mathbb{H}$. Indeed, otherwise, since $\mathbb{H}$ is closed with respect to addition, there exist $h_1,h_2 \in \mathbb{H} \subset K$ such that $h_1 + h_2 \in K $, which in turn, yields that $\mathbb{S}/K$ is not cancellative. So, the assumption $K \cap \mathbb{H} = \{ 1 \}$ in Proposition \ref{prop_cancellative_supplement}(1) excludes all cases for which $K \cap \mathbb{H} =L$ is a non-trivial kernel of $\mathbb{H}$ such that $\bar{\mathbb{H}} = \mathbb{H}/L$ is cancellative. In these cases, $\mathbb{S}/K$ is an extension $\bar{\mathbb{H}}(\bar{x})$ of $\bar{\mathbb{H}}$ which can be considered in the same way replacing $\mathbb{H}$ by $\bar{\mathbb{H}}$ and $D(\mathbb{H})$ by $D(\bar{\mathbb{H}})$. On the other hand, $\mathbb{H}$ has no non-trivial kernels $L$ such that $\mathbb{H}/L$ is cancellative if and only if $D(\mathbb{H})$ is a field (by Proposition \ref{cor6_2}). The latter is the case of interest for us, and it implies that $D(\mathbb{H}[\bar{x}])$ and $D(\mathbb{S}/K)$ coincide.
\end{rem}

\ \\
\ \\

\begin{cor} \label{maincor}
Let $\mathbb{H}$ be a commutative cancellative semifield such that $\mathbb{F} = D(\mathbb{H})$ is a field. Let $\mathbb{S}=\mathbb{H}(d)$ be a simple proper algebraic semifield extension of $\mathbb{H}$, i.e., $d$ is algebraic over $\mathbb{H}$. Then
$$\mathbb{H}(d) \cong \mathbb{H} + \mathbb{H} \tilde{x} + ... + \mathbb{H} \tilde{x}^{\ n-1}$$
where $\tilde{x} = x + \langle m(x) \rangle $  with $m(x) \in \mathbb{F}[x] \setminus \mathbb{H}[x]$ the minimal polynomial of $d$ over $\mathbb{F}$ and $n = deg(m(x))$.
\end{cor}
\begin{proof}
By our assumption that $\mathbb{H}(d)$ is a \emph{proper} semifield extension, we have that any polynomial $g(x) \in D(\mathbb{H})[x]$, such that $g(d)=0$ admits that $g(x) \not \in \mathbb{H}[x]$, i.e. $g(x)$ is of the form $g(x) = h_1(x) - h_2(x)$ for $h_1(x),h_2(x) \in \mathbb{H}[x]$ where $h_1(x) - h_2(x) \not \in \mathbb{H}[x]$. Now, since $D(\mathbb{H})$ is a field, by statement (2) of Proposition  \ref{prop_cancellative_supplement}, using the notation of the corollary, the ideal $A'$ is just $\langle m(x) \rangle$, where $m(x)$ is the minimal polynomial of $d$ over $D(\mathbb{H})$. The isomorphism $\Phi : D(\mathbb{H})[x]/A' \rightarrow D(\mathbb{H})[\bar{x}]$ of rings is induced by the map sending a polynomial $p(x) \in D(\mathbb{H})[x]  \ \mapsto p(\bar{x})$.
Now, $$D(\mathbb{H})[x]/A' = D(\mathbb{H})[x]/ \langle m(x) \rangle =  D(\mathbb{H}) + D(\mathbb{H}) \tilde{x} + ... + D(\mathbb{H}) \tilde{x}^{n-1},$$
where $\tilde{x} = x + A'$ and $n = deg(m(x))$. Thus, using $\Phi$, we get that $D(\mathbb{H})[\bar{x}]$ is the set of elements of the form $\Phi\left( \sum_{i=0}^{n-1}\alpha_{i} \tilde{x}^{i}\right) = \sum_{i=0}^{n-1}\alpha_{i} \bar{x}^{i}$, i.e.
$$D(\mathbb{H})[\bar{x}] = D(\mathbb{H}) + D(\mathbb{H}) \bar{x} + ... + D(\mathbb{H}) \bar{x}^{\ n-1}$$
where $\bar{x} = xK$ and $K$ the kernel corresponding to $A' = \langle m(x) \rangle$ and $n = deg(m(x))$. $D(\mathbb{H})[x]/\langle m(x) \rangle \cong D(\mathbb{H})[\bar{x}]$ and $D(\mathbb{H})[x]/\langle m(x) \rangle$ is a field, thus so is $D(\mathbb{H})[\bar{x}]$. Since $D(\mathbb{H})[\bar{x}] = D(\mathbb{H}[\bar{x}])$, we have that $\mathbb{H}[\bar{x}] = \mathbb{H} + \mathbb{H} \bar{x} + ... + \mathbb{H} \bar{x}^{n-1} $ is either a proper semifield or a field. The second cannot be, since $\mathbb{H}[\bar{x}] \subset \mathbb{H}(\bar{x})$ and $\mathbb{H}(\bar{x})$ is a proper semifield extending $\mathbb{H}$. Moreover, as $\mathbb{H}(\bar{x})$ is the semifield of fractions of $\mathbb{H}[\bar{x}]$, $\mathbb{H}[\bar{x}]$ being a semifield yields that $\mathbb{H}[\bar{x}] =  \mathbb{H}(\bar{x}) \cong \mathbb{H}(d)$. Using $$\Phi^{-1}(\mathbb{H} + \mathbb{H} \bar{x} + ... + \mathbb{H} \bar{x}^{n-1}) = \mathbb{H} + \mathbb{H} \tilde{x} + ... + \mathbb{H} \tilde{x}^{n-1},$$ we obtain the desired result.
\end{proof}

\begin{rem}
In the setting of Corollary \ref{maincor}, we can use the correspondence
$$\forall \ a(x), b(x) \in \mathbb{H}[x] \ : \ \frac{a(x)}{b(x)} \in K \Leftrightarrow a(x) - b(x) \in \langle m(x) \rangle $$
for understanding the form of the kernel $K$.
We have that $\frac{a(x)}{b(x)} \in K$ if and only if $a(x)-b(x) = m(x)g(x)$ for some $g(x) \in \mathbb{F}[x]$.
For any decomposition $$m(x) = m_{1}(x) - m_{2}(x) \ , \ \ g(x) = g_{1}(x) - g_{2}(x)$$
with $m_{1}(x), m_{2}(x),g_{1}(x), g_{2}(x) \in \mathbb{H}[x]$, we can write
$$m(x)g(x) = (m_{1}(x)g_{1}(x) + m_{2}(x)g_{2}(x)) - (m_{1}(x)g_{2}(x) + m_{2}(x)g_{1}(x)).$$
Thus, using the opposite direction of the correspondence yields that $$\frac{a(x)}{b(x)} = \frac{m_{1}(x)g_{1}(x) + m_{2}(x)g_{2}(x)}{m_{1}(x)g_{2}(x) + m_{2}(x)g_{1}(x)}.$$
Now, let $m(x) = m_{+}(x) - m_{-}(x)$ with $m_{+}(x), m_{-}(x) \in \mathbb{H}[x]$ be the unique \linebreak decomposition such that there is no $s_1(x),s_2(x),h(x) \in \mathbb{H}[x]$ such that $$m_{+}(x)=s_1(x)+h(x) \  \text{and}  \  m_{-}(x)=s_2(x) +h(x).$$ This decomposition is obtained, grouping all the monomials in $\mathbb{H}[x]$ and all the \linebreak monomials in $-\mathbb{H}[x]$ of $m(x)$ (note that this decomposition can be always obtained as $m(x) \in D(\mathbb{H})[x]$). Thus, by the above, we have $$m_1(x) = m_{+}(x) + h(x) \ \text{and}  \ m_{2}(x) = m_{-}(x) +h(x)$$ for some $h(x) \in \mathbb{H}[x]$, yielding that
$$\frac{a(x)}{b(x)} = \frac{m_{+}(x)g_{1}(x) + m_{-}(x)g_{2}(x) + h(x)(g_1(x)+g_2(x))}{m_{+}(x)g_{2}(x) + m_{-}(x)g_{1}(x)  + h(x)(g_1(x)+g_2(x))}.$$
Now, since $h(x),g_1(x)$ and $g_2(x)$ vary over all elements of $\mathbb{H}[x]$, we get that
\begin{align*}
K = \left\{ \frac{m_{+}(x)g_{1}(x) + m_{-}(x)g_{2}(x) + h(x)(g_1(x)+g_2(x))}{m_{+}(x)g_{2}(x) + m_{-}(x)g_{1}(x)  + h(x)(g_1(x)+g_2(x))} \ : \ g_1,g_2,h \in \mathbb{H}[x] \right\}.
%= \\
%\left\{ \frac{m_{+}(x)g_{1}(x) + m_{-}(x)g_{2}(x) + 1}{m_{+}(x)g_{2}(x) + m_{-}(x)g_{1}(x)  + 1} \right\} \cup  \left\{ %\frac{m_{+}(x)g_{1}(x) + m_{-}(x)g_{2}(x)}{m_{+}(x)g_{2}(x) + m_{-}(x)g_{1}(x)}  \right\}
\end{align*}
%where $g_1(x)$ and $g_2(x)$ vary over all elements of $\mathbb{H}[x]$.
\end{rem}

%
%\begin{cor} \label{cor2}
%Let $\mathbb{H}$ be a commutative cancellative semifield such that $\mathbb{F} = D(\mathbb{H})$ is a field. Let $S=\mathbb{H}(d)$ be a simple algebraic extension of $\mathbb{H}$, i.e. $d$ is algebraic over $\mathbb{H}$. Then, by corollary \ref{cor1} $\mathbb{F}[x]/A \cong \mathbb{F}[\bar{x}] = D(\mathbb{H}[\bar{x}])$ with $A = g(x)\mathbb{F}[x]$ where $g(x) \in \mathbb{F}[x]$ is the minimal polynomial of $d$ over $\mathbb{F}$ and $\bar{x} = xK$. By field theory we have that $\mathbb{F}[d] = \mathbb{F}[x]/A$ is a field, thus so is $D(\mathbb{H}[\bar{x}])$. Now, $\mathbb{H}[\bar{x}]=MESF(D(\mathbb{H}[\bar{x}]))$ is the proper semifield of $D(\mathbb{H}[\bar{x}])$ and thus $\mathbb{H}[\bar{x}] = \mathbb{H}(\bar{x}) \cong \mathbb{H}(d)$. On the other hand, $\mathbb{F}[x]/A = \mathbb{F} + \bar{x}\mathbb{F} + ... + \bar{x}^n\mathbb{F} = D(\mathbb{H}) + \bar{x}D(\mathbb{H}) + ... + \bar{x}^nD(\mathbb{H})$ with $n = deg(g)$. Thus $MESF(\mathbb{F}[x]/A) = \mathbb{H} + \bar{x}\mathbb{H} + ... + \bar{x}^n\mathbb{H}$. So $\mathbb{H}(d) \cong  \mathbb{H} + \bar{x}\mathbb{H} + ... + \bar{x}^n\mathbb{H}$.
%\end{cor}

%
%\begin{note}
%Note that the additive cosets of the extension $\mathbb{F}[d]$ of $\mathbb{F}$ are of the form $x+A$ whereas the cosets
%of the extension $\mathbb{F}[\bar{x}]$ are the multiplicative cosets of the form $xK$, where $A$ corresponds to $K$ by the theory introduced above.
%\end{note}

In view of Corollary \ref{maincor}, we can introduce the following definition:
\begin{defn}
Let $\mathbb{H}$ be a commutative cancellative semifield such that $\mathbb{F} = D(\mathbb{H})$ is a field. Let $\mathbb{S}=\mathbb{H}(d)$ be a simple algebraic extension of $\mathbb{H}$. We define the \emph{dimension} of \ $\mathbb{S}$ over $\mathbb{H}$, $$dim_{\mathbb{H}}\mathbb{S} = [\mathbb{S} : \mathbb{H}] = [D(\mathbb{H})[d] : D(\mathbb{H})]=deg(g)$$
where $g(x) \in D(\mathbb{H})[x]$ is the minimal polynomial of $d$ over $\mathbb{F}$.
\end{defn}

\begin{rem}
In the case where $d$ is transcendental, $\mathbb{H}[d] \subset \mathbb{H}(d) \cong \mathbb{H}(x)$ and $[\mathbb{H}[d]:\mathbb{H}]=\infty$, for otherwise, if $[\mathbb{H}[d]:\mathbb{H}]= n <\infty$ we get that $1,d,...,d^n$ are linearly dependent over $D(\mathbb{H})$ yielding a polynomial $f \in D(\mathbb{H})[x]$ such that $f(d)=0$, contradicting the assumption that $d$ is transcendental.
\end{rem}

\begin{cor}
The observations introduced above yield that for $\mathbb{H}$, a commutative cancellative \emph{proper} semifield, such that $D(\mathbb{H})$ is a field, the theory of proper cancellative affine semifield extensions of $\mathbb{H}$ is analogous to the theory of affine field extensions. In particular, one gets the transcendence degree and the degree of an \emph{affine} algebraic extension. The only difference is that field extensions of $\mathbb{H}$ are excluded. These extensions are characterized in Lemma \ref{field_lemma}.
\end{cor}

\ \\
\section{Uniform layered extensions} \label{section:UniformExtensions} \ \\

%\title{Uniform layered extensions}
%\maketitle

In this section, we consider uniform layered domains with a cancellative sorting semiring. Nevertheless, the theory developed below applies to any sorting semiring.\\

\begin{defn}\label{defn_uniform_layered_domain}
A uniform layered domain or a uniform $\mathcal{L}$-layered domain is the set of pairs $\mathbb{D} = \{ ^{[l]}a \ : \ l \in \mathcal{L}, \ a \in \mathcal{G} \}$
where $\mathcal{G}$ is an ordered semiring which is a domain and $\mathcal{L}$ is partially pre-ordered semiring without zero. We call $\mathcal{G}$ the `semiring of values' while $\mathcal{L}$ is called the `sorting semiring'.
We write $\mathcal{G}(\mathbb{D})$ and $\mathcal{L}_{\mathbb{D}}$ to indicate the semiring of values and the sorting semiring of the uniform layered domain $\mathbb{D}$.
The multiplication on $\mathbb{D}$ is defined componentwise, i.e.,
$$^{[k]}a ^{[l]}b = ^{[kl]}ab,$$
and addition by the rules
\begin{equation*}\label{eq_1_defn_tensor_strucature}
^{[k]}a ^{[l]}b = \begin{cases}
      ^{[k]}a   &  a > b; \\
      ^{[l]}b   & a < b; \\
      ^{[k+l]}b   & a = b.
      \end{cases}
\end{equation*}
With respect to these operations $\mathbb{D}$ is a semiring.\\
Define the sorting map $s: \mathbb{D} \rightarrow \mathcal{L}_{\mathbb{D}}$ by $$ s(^{[l]}a) = l$$
and the ghost map (we also call the value map) $\nu : \mathbb{D} \rightarrow \mathcal{G}(\mathbb{D})$ by $$\nu (^{[l]}a) = a.$$
The sorting map, $s$, is onto $\mathcal{L}_{\mathbb{D}}$ and admits $s(1_{\mathbb{D}}) = 1_{\mathcal{L}}$ and  $s(ab) = s(a) s(b)$ for all $a, b \in \mathbb{D}$.
The ghost map, $\nu$, is a semiring epimorphism of $\mathbb{D}$ onto $\mathcal{G}(\mathbb{D})$.
\end{defn}

\begin{nota}
We also denote the uniform layered domain $\mathbb{D}$ by $\mathcal{L} \odot \mathcal{G}$ where $\mathcal{L} = \mathcal{L}_{\mathbb{D}}$ and $\mathcal{G} = \mathcal{G}(\mathbb{D})$, and an element $^{[l]}a \in \mathbb{D}$ by $l \odot a$ where $l \in \mathcal{L}$ and $a \in \mathcal{G}$.
\end{nota}

\ \\

\ \\

In Definition \ref{defn_uniform_layered_domain}, we actually introduce a generic construction of a uniform layered domain which we present in the subsequent Theorem \ref{thm_uniform_decomposition}. For a more general definition and additional details regarding the construction, we refer the reader to \cite{Layered}.

\begin{thm}\label{thm_uniform_decomposition} \cite{Layered}
Let $\mathbb{H}$ be a uniform $\mathcal{L}_{\mathbb{H}}$-layered domain.
Then $$\mathbb{H} \cong s(\mathbb{H}) \odot \nu(\mathbb{H})= \mathcal{L}_{\mathbb{H}} \odot \mathcal{G}(\mathbb{H}).$$
\end{thm}
\begin{proof}
Define the map $\Phi : \mathbb{H} \rightarrow \mathcal{L}_{\mathbb{H}} \odot \mathcal{G}(\mathbb{H})$ by $\Phi(a) = s(a) \odot  \nu(a)$ for any $a \in \mathbb{H}$. Then by  Theorem 3.21 of \cite{Layered}, $\Phi$ is a semiring isomorphism.
\end{proof}

\begin{rem}
Let $\mathbb{D}$ be a uniform $\mathcal{L}_{\mathbb{D}}$-layered domain, and let $a \in \mathbb{D}$ be an element of $\mathbb{D}$. Then $s(a)\in \mathcal{L}_{\mathbb{D}}$ is said to be the \emph{layer} of $a$ and  $\nu(a)$ the \emph{$\nu$-value} or \emph{ghost value} of $a$. We also write $\nu(a)$ for the element $^{[1]}\nu(a) \in \mathbb{D}$. The distinction between the cases will be clear from the context in which it appears. Since $\mathbb{D}$ is uniform we have that $^{[s(a)]}1 \in \mathbb{D}$ where $1$ denotes the identity element with respect to multiplication. Thus we can write $a$ as $a = \ ^{[s(a)]}1 \cdot \nu(a)$. Notice that $$^{[s]}1 \cdot (\nu(a)+\nu(b)) = \ ^{[s]}1 \cdot \nu(a) + \ ^{[s]}1 \cdot \nu(b) , \ \ \
(\ ^{[s]}1 + \ ^{[t]}1)  \cdot \nu(a) = \ ^{[s]}1 \cdot \nu(a) + \ ^{[t]}1  \cdot \nu(a).$$
Indeed, the right equality is straightforward since
$$(^{[s]}1 + \ ^{[t]}1)  \cdot \nu(a) = \ ^{[s+t]}1 \cdot \nu(a) =  ^{[s]}1 \cdot \nu(a) + \ ^{[t]}1  \cdot \nu(a).$$
For the left equality, if $\nu(a) > \nu(b)$, then also $^{[s]}1 \cdot \nu(a) > \ ^{[s]}1 \cdot \nu(b)$ yielding that
$^{[s]}1 \cdot (\nu(a)+\nu(b)) = \ ^{[s]}1 \cdot (\nu(a)) = \ ^{[s]}1 \cdot \nu(a) + \ ^{[s]}1 \cdot \nu(b).$
Otherwise, if $\nu(a) = \nu(b)$ then also $^{[s]}1 \cdot \nu(a) > \ ^{[s]}1 \cdot \nu(b)$ and thus
$$^{[s]}1 \cdot (\nu(a)+\nu(b)) = \ ^{[s]}1 \cdot (^{[2]}\nu(a)) = \ ^{[2s]}1 \cdot (\nu(a)) = \ ^{[s]}1 \cdot \nu(a)+ \ ^{[s]}1 \cdot \nu(b).$$
\end{rem}
%
%where, for any $a \in \mathbb{A}$ and $b \in \mathbb{B}$, $a \odot b$ is defined to be an additively bilinear product, in the sense that

A result introduced in \cite{Layered} is

\begin{prop}\label{L_G_decomp}
A uniform pre-domain $\mathbb{H}$ is a uniform semifield if and only if $\mathcal{G}(\mathbb{H})$ and $\mathcal{L}_{\mathbb{H}}$ are both semifields.
\end{prop}

\begin{rem}
Note that the $+$ and $\cdot$ operations of $\mathbb{H}$ induce the $max$ and the classic addition operations on $\mathcal{G}(\mathbb{H})$, while the operations on $\mathcal{L}_{\mathbb{H}}$ (restricting to any given $\nu$-value) are the classic addition and multiplication, respectively. $\mathbb{H}$ is said to be an \emph{1-semifield} if we only require $\mathcal{G}(\mathbb{H})$ to be a semifield.
\end{rem}
\bigbreak

\begin{rem}
Let $\mathbb{H}$ be an uniform $\mathcal{L}_{\mathbb{H}}$-layered domain, where $\mathcal{L}_{\mathbb{H}}$ is a semifield.
Let $\mathbb{D}$ be an uniform $\mathcal{L}_{\mathbb{D}}$-layered domain extending $\mathbb{H}$ with $\mathcal{L}_{\mathbb{D}}$ the sorting domain of $\mathbb{D}$. Since $\mathbb{D}$ contains $\mathbb{H}$ we have that $\mathcal{L}_{\mathbb{D}} \supseteq \mathcal{L}_{\mathbb{H}}$.
\end{rem}

more generally we can say that

\begin{defn}
We define a uniform sub-domain $\mathbb{E}$ of a layered domain $\mathbb{D}$ to be a domain of the form $\mathcal{L}_{\mathbb{E}} \odot \mathcal{G}(\mathbb{E})$, where $\mathcal{L}_{\mathbb{E}} \subseteq \mathcal{L}_{\mathbb{D}}$ and $\mathcal{G}(\mathbb{E}) \subseteq \mathcal{G}({\mathbb{D}})$.
\end{defn}

\begin{note}
Throughout this section, when not stated otherwise, $\mathbb{H}$ will always denote a uniform $\mathcal{L}_{\mathbb{H}}$-layered semifield, and $\mathbb{D}$ will always denote a uniform $\mathcal{L}_{\mathbb{D}}$-layered domain extending $\mathbb{H}$.
\end{note}
\begin{rem}
Let $f(x) \in \mathbb{H}[x]$ and let $\mathbb{D}$ be a layered domain extending $\mathbb{H}$. Then, for any  $a \in \mathbb{D}$
\begin{equation}\label{eq5}
\nu(f(a)) = \nu(f(\nu(a))).
\end{equation}
\end{rem}
\begin{proof}
A  straightforward consequence of of Definition \ref{defn_uniform_layered_domain}.
\end{proof}

\begin{rem}
Let $f(x) \in \mathbb{H}[x]$. Write $f(x) = \sum_{i=0}^{m}\alpha_i x^i$ where $\alpha_i \in \mathbb{H}$. Then for any $a \in \mathbb{D}$ the following hold:
\begin{equation}\label{eq3}
s(f(a))  \in \mathcal{L}_{\mathbb{H}}[s(a)] = \left\{g(s(a)) \ : \ g \in  \mathcal{L}_{\mathbb{H}}[x] \right\}.
\end{equation}
\begin{equation}\label{eq4}
f(a) = f(s(a) \odot \nu(a)) = s(f(a)) \odot \nu(f(a)).
\end{equation}
\end{rem}
\begin{proof}
For the first equality, (\ref{eq3}), since
$$s(f(a)) = s\bigg(\sum_{i=0}^{m}\alpha_i a^i\bigg) = \sum_{j \in J}s(\alpha_j) s(a)^i = \tilde{f}\big(s(a)\big).$$
Here $J \subseteq \{0,...,m\}$ corresponds to the set of indices of dominant (essential) terms of $\sum_{i=0}^{m}\alpha_i a^i$, where $\nu(\alpha_{j_1} a^{j_1}) = \nu(\alpha_{j_2}a^{j_2})$ for any pair of essential terms indexed by $j_1$ and $j_2$ in $J$. $\tilde{f} \in \mathcal{L}_{\mathbb{H}}[x]$ is a polynomial with coefficients in $\mathcal{L}_{\mathbb{H}}$, determined by the dominant (essential) part of $f$. So, we have that $s(f(a)) \in \mathcal{L}_{\mathbb{H}}[s(a)]$ as desired.\\ The second equality, (\ref{eq4}), is a direct consequence of Definition \ref{defn_uniform_layered_domain}.
\end{proof}
\bigbreak

%\begin{rem}\label{rem_uniform_closure}
%For any  layered domain $\mathbb{D}$, let
%\begin{equation}\label{eq_rem_uniform_closure}
%L_{\mathbb{D}} = \{ s(a)  \ : \ a \in \mathbb{D} \} \ \ \text{and} \ \  V_{\mathbb{D}} = \{ \nu(a) \ : \  a \in \mathbb{D} \}.
%\end{equation}
%Then $L_{\mathbb{D}} \odot V_{\mathbb{D}}$ is a uniform layered domain and it is contained in
%any uniform layered domain $\mathbb{K}$ containing $\mathbb{D}$.
%\end{rem}
%\begin{proof}
%Denote $\mathbb{E} = L_{\mathbb{D}} \odot V_{\mathbb{D}}$. Then $\mathbb{E}$ is a uniform layered semiring by definition. Since $\mathbb{D}$ is a domain so is $\mathbb{E}$. Indeed, as the $\nu$-values of $\mathbb{E}$, $V_{\mathbb{D}}$,  are completely determined by the $\nu$-values of $\mathbb{D}$ a zero divisor in $\mathbb{E}$ gives rise to a zero devisor in $\mathbb{D}$ taking elements in $\mathbb{D}$ having the same $\nu$-values (such exist by \eqref{eq_rem_uniform_closure}). By definition, for any uniform layered domain $\mathbb{K}$ containing $\mathbb{D}$ both $L_{\mathbb{D}} \subseteq \mathcal{L}_{\mathbb{K}}$ and $V_{\mathbb{D}} \subseteq \mathcal{G}(\mathbb{K})$ thus by Theorem \ref{thm_uniform_decomposition} $L_{\mathbb{D}} \odot V_{\mathbb{D}} \subseteq \mathcal{L}_{\mathbb{K}} \odot \mathcal{G}(\mathbb{K}) = \mathbb{K}$.
%\end{proof}
%
\begin{defn}
Let $\mathbb{D}$ be a layered domain and let $E \subset \mathbb{D}$. Define $\mathscr{U}_{\mathbb{D}}(E)$ to be a minimal (with respect to inclusion) uniform layered domain $\mathbb{E}$ such that $E \subseteq \mathbb{E} \subseteq \mathbb{D}$, if such a minimum exists. In case it is defined and unique up to isomorphism,   $\mathscr{U}_{\mathbb{D}}(E)$ is said to be the \emph{uniform closure} of $E$ in $\mathbb{D}$.
\end{defn}
\bigbreak

\begin{defn}
Let  $\mathbb{D}$ a layered domain extending a layered semifield $\mathbb{H}$. Let $a \in \mathbb{D}$ be an element of $\mathbb{D}$. Define $$\mathbb{H}[a]= \left\{ f(a) \ : \ f \in \mathbb{H}[x] \right\}.$$
%Then $\mathscr{U}(\mathbb{H}[a])$ is the uniform closure of $\mathbb{H}[a]$.
\end{defn}

\bigbreak

\begin{rem}
By equation \eqref{eq4} we have that
$$\mathbb{H}[a] = \left\{s(f(a)) \odot \nu(f(a)) \ : \ f \in \mathbb{H}[x] \right\}.$$
\end{rem}

\bigbreak

\begin{prop}\label{layering_ext_prop}
Let $\mathbb{H}[a]= \left\{ f(a) \ : \ f \in \mathbb{H}[x] \right\}$ where $a \in \mathbb{D}$ such that $\nu(a)~\in~\mathcal{G}(\mathbb{H})$. Then
$\mathscr{U}_{\mathbb{D}}\left(\mathbb{H}[a]\right) = \mathbb{H}[a] =  \mathcal{L}_\mathbb{H}[s(a)] \odot \mathcal{G}(\mathbb{H})$.
\end{prop}
\begin{proof}
By equations \eqref{eq5} and \eqref{eq4} we have that
$$f(a) = f(s(a) \odot \nu(a)) = s(f(a)) \odot \nu(f(a)) =  s(f(a)) \odot \nu(f(\nu(a))).$$ Now, $\nu(a) \in \mathcal{G}(\mathbb{H})$ thus $f(\nu(a)) \in \mathbb{H}$ which yields that $\nu(f(\nu(a))) \in  \mathcal{G}(\mathbb{H})$. By \eqref{eq3} we have that $s(f(a))\in \mathcal{L}_{\mathbb{H}}[s(a)]$. Thus $\mathbb{H}[a] \subseteq \mathcal{L}_\mathbb{H}[s(a)] \odot \mathcal{G}(\mathbb{H})$.\\
Let $g(s(a)) \odot \nu(\alpha) \in \mathcal{L}_\mathbb{H}[s(a)] \odot \mathcal{G}(\mathbb{H})$. Write $g(x) =\sum_{i=0}^{m}s_ix^i$ with $s_i \in \mathcal{L}_{\mathbb{H}}$ and then
define $f(x) = \sum_{i=0}^{m}{^{[s_i]}\left(\nu(\alpha)\nu(a)^{\ -i}\right) x^i}$. Since $\nu(\alpha)$ and $\nu(a)$ are in $\mathcal{G}(\mathbb{H})$, we have that $\nu(\alpha)\nu(a)^{\ -i} \in \mathcal{G}(\mathbb{H})$ for each $i = 0,...,m$ and thus $^{[s_i]}\nu(\alpha)\nu(a)^{\ -i} \in \mathbb{H}$. So $f(x)\in \mathbb{H}[x]$.
Now,
$$\nu\left(^{[s_i]}\left(\nu(\alpha)\nu(a)^{\ -i}\right)a^{i}\right)= \nu(\alpha)\nu(\nu(a)^{\ -i})\nu(a^{i}) = \nu(\alpha)\nu(a^{\ -i})\nu(a^{i}) =\nu(\alpha)$$
for every $i = 0,...,m$. So the following equalities hold:

\begin{align*}
\nu(f(a)) \ = & \nu\left(\sum_{i=0}^{m}{ ^{[s_i]}\left(\nu(\alpha)\nu(a)^{\ -i}\right) a^i}\right) = \nu\left(\sum_{i=0}^{m}\nu\Big(^{[s_i]}\left(\nu(\alpha)\nu(a)^{\ -i}\right)a^i\Big)\right)\\
\ = & \nu\left(\sum_{i=0}^{m}\nu(\alpha)\right)=\nu(\alpha) \end{align*}

\begin{align*} s(f(a)) \ = & s\left(\sum_{i=0}^{m}{^{[s_i]}\left(\nu(\alpha)\nu(a)^{\ -i}\right) a^i}\right)= \sum_{i=0}^{m}s\left(^{[s_i]}\left(\nu(\alpha)\nu(a)^{\ -i}\right)\right)s(a)^i\\
 \ = & \sum_{i=0}^{m}s_i s(a)^i= g(s(a)).\end{align*}
Thus $g(s(a)) \odot \nu(\alpha) = s(f(a)) \odot \nu(f(a))= f(a)\in \mathbb{H}[a]$. \\
Finally, since $\mathbb{H}[a] = \mathcal{L}_\mathbb{H}[s(a)] \odot \mathcal{G}(\mathbb{H})$ is uniform, we have that
$\mathscr{U}_{\mathbb{D}}(\mathbb{H}[a]) = \mathbb{H}[a]$.
\end{proof}

\begin{prop}\label{nu_ext_prop}
Let \ $\mathbb{H}[a]= \{ f(a) \ : \ f \in \mathbb{H}[x] \}$ where $a \in \mathbb{D}$ is such that $s(a)~\in~\mathcal{L}_{\mathbb{H}}$. Then
$\mathscr{U}_{\mathbb{D}}(\mathbb{H}[a]) = \mathbb{H}[a] = \mathcal{L}_\mathbb{H} \odot \mathcal{G}(\mathbb{H}[\nu(a)])$.
\end{prop}
\begin{proof}
By equations \eqref{eq5} and \eqref{eq4}, we have that
$$f(a) = f(s(a) \odot \nu(a)) = s(f(a)) \odot \nu(f(a)) =  s(f(a)) \odot \nu(f(\nu(a))).$$
Now, by \eqref{eq3}, $s(f(a))  \in \mathcal{L}_{\mathbb{H}}[s(a)]$. Since $s(a) \in \mathcal{L}_{\mathbb{H}}$ we have that $\mathcal{L}_{\mathbb{H}}[s(a)] = \mathcal{L}_{\mathbb{H}}$, thus $s(f(a))  \in \mathcal{L}_{\mathbb{H}}$. Since  $\mathbb{H}[\nu(a)]=\left\{ f(\nu(a)) \ : \ f \in \mathbb{H}[x] \right\} $, we have that $f(\nu(a)) \in \mathbb{H}[\nu(a)]$ implying that $\nu(f(\nu(a))) \in \mathcal{G}(\mathbb{H}[\nu(a)])$. Thus $\mathbb{H}[a] \subseteq \mathcal{L}_\mathbb{H} \odot \mathcal{G}(\mathbb{H}[\nu(a)])$. Conversely, let $s \odot \nu(g(\nu(a))) \in \mathcal{L}_\mathbb{H} \odot \mathcal{G}(\mathbb{H}[\nu(a)])$. Write $g = \sum_{i=0}^{n}\alpha_i x^i$ where $\alpha_i \in \mathbb{H}$ for each $i=0,...,n$. Let ${i_0,...,i_k}$, $k \leq n$ be the indices corresponding to the essential terms of $g(\nu(a))$. Define $f(x) = \sum_{j=0}^{k}{ ^{[s/(\sum{s(\alpha_{i_j})})]}\alpha_{i_j}x^{i_j}}$. Then
\begin{align*} s(f(a)) \ = & s\left( \sum_{j=0}^{k}{ ^{[s/\sum{s(\alpha_{i_j}})]}\alpha_{i_j}a^{i_j}}\right) =
\sum_{j=0}^{k} (s/\sum{s(\alpha_{i_j})}) s(\alpha_{i_j})s(a^{i_j})\\
\ = & (s/\sum{s(\alpha_{i_j})})\sum_{j=0}^{k} s(\alpha_{i_j}) = s,    \end{align*}

\begin{align*} \nu(f(a)) \ = & \nu\left( \sum_{j=0}^{k}{ ^{[s/(\sum{s(\alpha_{i_j})})]}\alpha_{i_j}a^{i_j}}\right) =
\nu\left(\sum_{j=0}^{k} \alpha_{i_j} \nu(a)^{i_j}\right) \\
= & \nu\left(\sum_{i=0}^{n} \alpha_{i}\nu(a)^{i}\right) =  \nu (g(\nu(a)).
    \end{align*}

In the first equality, we may sum up the layers of the terms since by assumption they all have the same $\nu$-value.
In the second equality, we use \eqref{eq5} in the first step \linebreak calculation.
Thus $s \odot \nu(g(\nu(a))) = s(f(a)) \odot \nu(f(a)) = f(a) \in \mathbb{H}[a]$, as required.
Finally, since $\mathbb{H}[a]=\mathcal{L}_\mathbb{H} \odot \mathcal{G}(\mathbb{H}[\nu(a)])$ is uniform, we have that
$\mathscr{U}_{\mathbb{D}}(\mathbb{H}[a]) = \mathbb{H}[a]$.
\end{proof}

\begin{rem}
For $f \in \mathbb{H}[x]$ and $a \in \mathbb{D}$, by equation \eqref{eq3}, we have that \linebreak  $\nu(f(\nu(a))) = \nu(f(a))$. Thus $\mathcal{G}(\mathbb{H}[\nu(a)])=\mathcal{G}(\mathbb{H}[a])$. Consequently, Proposition \ref{nu_ext_prop} asserts that
$\mathscr{U}_{\mathbb{D}}(\mathbb{H}[a]) = \mathbb{H}[a] = \mathcal{L}_\mathbb{H} \odot \mathcal{G}(\mathbb{H}[a])$.
\end{rem}

\bigbreak

\begin{defn}
We call the extensions introduced in Propositions \ref{layering_ext_prop} and \ref{nu_ext_prop} \emph{pure-layer} extension and \emph{pure-value} extension, respectively.
\end{defn}

By Propositions \ref{layering_ext_prop} and \ref{nu_ext_prop} we have
\begin{prop}\label{pure_ext_prop}
Pure-layer extensions and pure-value extensions are uniform extensions.
\end{prop}

\begin{defn}
Let $\mathbb{H}$ be a layered semiring (not necessarily uniform). Let $\alpha \in \mathcal{G}(\mathbb{H})$ be any $\nu$-value of $\mathbb{H}$. We call the set $$L_\alpha \doteq \{s(a) \ : \  a \in \mathbb{H}, \ \nu(a) = \alpha \}$$
the \emph{layer fibre} of $\alpha$.
\end{defn}

\begin{flushleft}When considering Proposition \ref{pure_ext_prop}, the following natural question arises:\\ \end{flushleft}
\textbf{Under what conditions, is an extension of the form $\mathbb{H}[a]$, with $a \in \mathbb{D}$ and $\mathbb{H}$ a \linebreak uniform domain, a uniform extension?} \\

\begin{flushleft} In general, for an extension of the form $\mathbb{H}[a]$, the layer fibers of different $\nu$ values of $\mathbb{H}[a]$ may differ. It is even possible that $L_\alpha \neq \mathcal{L}_\mathbb{H}[s(a)]$ for any $\alpha \in \mathcal{G}(\mathbb{H}[\nu(a)])$.\end{flushleft}

\begin{rem}\label{uniform_rem}
In the notation introduced above, let $\mathbb{K}$ be a layered domain and let $\beta \in \mathcal{G}(\mathbb{K})$ be an invertible element. Then $L_\beta \subset L_\alpha$ for any $\alpha \in \mathcal{G}(\mathbb{K})$. In particular, if $\mathcal{G}(\mathbb{K})$ is a semifield, then $L_\alpha = L_\beta$ for any $\alpha, \beta \in \mathcal{G}(\mathbb{K})$.
\end{rem}
\begin{proof}
 Let $l \in L_\beta$ be any layer in the layer fiber of $\beta$. Then there exists $b \in \mathbb{K}$ such that $\nu(b) = \beta$ and $s(b) = l$. Then $\nu(\beta^{-1}b) = \nu(\beta^{-1})\nu(b) =\beta^{-1}\beta =  1_{\mathcal{G}(\mathbb{K})}$ and $s(\beta^{-1}b)=s(\beta^{-1})s(b) = 1_{\mathcal{L}_{\mathbb{K}}}s(b) = s(b)$. Now, consider the element $\alpha \beta^{-1}b \in \mathbb{K}$. We have that $\nu(\alpha \beta^{-1}b)= \nu(\alpha)\nu(\beta^{-1}b)=\alpha 1_{\mathcal{G}(\mathbb{K})}=\alpha$ and $s(\alpha \beta^{-1}b)= s(\alpha)s(b)=1_{\mathcal{L}_{\mathbb{K}}}s(b)=s(b)$. So $l=s(b) \in L_\alpha$ and thus $L_\beta \subset L_\alpha$.
\end{proof}
%
%\begin{rem}\label{uniform_rem}
%In the notation introduced above, let $\beta \in \mathcal{G}(\mathbb{H}[a])$ be an invertible element, then $L_\beta \subset L_\alpha$ for any $\alpha, \in \mathcal{G}(\mathbb{H}[a])$. In particular, if $\mathcal{G}(\mathbb{H}[a])$ is a semifield, then $L_\alpha = L_\beta$ for any $\alpha, \beta \in \mathcal{G}(\mathbb{H}[a])$.
%\end{rem}
%\begin{proof}
%for any $b \in \mathbb{H}[a]$ such that $\nu(b) = \beta$, since $s(\alpha \beta^{-1}b)= s(b)$ we have that $L_\beta \subset L_\alpha$.
%\end{proof}

We will next characterize the simple uniform layered semifield extensions, i.e., extensions of the form $\mathbb{H}[a]$ with $\mathbb{H}$ a uniform semifield.
\bigbreak

\begin{prop}
Pure-layer and pure-value extensions of a uniform semifield $\mathbb{H}$ are the only simple uniform layered semifield extensions.
\end{prop}

\begin{proof}
Let $\mathbb{H}$ be a layered semifield and $\mathbb{D}$ a layered domain extending $\mathbb{H}$. Let $a \in \mathbb{D}$.
%For simplicity of notation, we write $a = (s(a),\nu(a))$ where $s(a) \in \mathcal{L}_{\mathbb{D}}$ is the layer of $a$ and $\nu(a) \in \mathcal{G}(\mathbb{D})$ is the $\nu$-value of $a$. We will use the same notation for any element of $\mathbb{D}$ (in particular of $\mathbb{H}$).\\
Consider a polynomial $p(x) = \sum_{k=0}^{n}\alpha_k x^k \in \mathbb{H}[x]$, then
$$p(a)= \sum_{k=0}^{n}\alpha_k a^k  = \sum_{k=0}^{n}(s(\alpha_k)s(a)^k \odot \nu(\alpha_k)\nu(a)^k).$$
\bigbreak
First, note that by definition $\mathcal{L}_{\mathbb{H}[a]} \subseteq \mathcal{L}_{\mathbb{H}}[s(a)]$. Taking constant polynomials yields that $\mathcal{L}_{\mathbb{H}} \subset \mathcal{L}_{\mathbb{H}[a]}$. Moreover, taking monomials in $\mathbb{H}[x]$, one sees at once that
$$A = \left\{ s(\alpha) s(a)^k \ : \ \alpha \in \mathbb{H}, \ k \in \mathbb{N} \right\} =  \bigcup_{k \in \mathbb{N}} \mathcal{L}_{\mathbb{H}}\cdot s(a)^k \subset \mathcal{L}_{\mathbb{H}[a]}.$$
Now, if $s(a) \not \in \mathcal{L}_{\mathbb{H}}$, sums of elements of $A$ are not guaranteed to be in $\mathcal{L}_{\mathbb{H}[a]}$. Moreover, we claim that if $\nu(a)$ is not $\mathcal{G}(\mathbb{H})$-torsion, no proper sum of elements of $A$ is in
$\mathcal{L}_{\mathbb{H}[a]}$. Indeed, for a proper sum $s(\alpha) s(a)^{k_1} + s(\beta)s(a)^{k_2}$ with $\alpha, \beta \in \mathbb{H}$, $k_1 \neq k_2 \in \mathbb{N}$ to be in $\mathcal{L}_{\mathbb{H}[a]}$, the $\nu$-values $\nu(\alpha)\nu(a)^{k_1}$ and $\nu(\beta)\nu(a)^{k_2}$ must coincide, implying that $\nu(a)$ is $\mathcal{G}(\mathbb{H})$-torsion. We conclude that if $s(a) \not \in \mathcal{L}_{\mathbb{H}}$ and $\mathcal{L}_{\mathbb{H}[a]}  = \mathcal{L}_{\mathbb{H}}[s(a)]$ then $\nu(a)$ is
$\mathcal{G}(\mathbb{H})$-torsion. The converse, however, is not true.  If $\nu(a)$ is $\mathcal{G}(\mathbb{H})$-torsion with $r = Rank_{\mathcal{G}(\mathbb{H})}(\nu(a)) > 0$ (i.e., $\nu(a) \not \in \mathcal{G}(\mathbb{H})$) then $\mathcal{L}_{\mathbb{H}[a]} \neq \mathcal{L}_{\mathbb{H}}[s(a)]$. Indeed, reversing the arguments introduced above, one sees that for any $k_1 > k_2 + r$, $\nu(\alpha)\nu(a)^{k_1}$ and $\nu(\beta)\nu(a)^{k_2}$ cannot coincide, due to the minimality of $r$.
\end{proof}

\bigbreak

\begin{flushleft}In view of these assertions we have the following result:\end{flushleft}

\begin{cor}\label{regular_extension_cor}
Let $\mathbb{H}$ be a layered semifield and $\mathbb{D}$ a layered domain extending $\mathbb{H}$. Let $a \in \mathbb{D}$. Then $\mathcal{L}_{\mathbb{H}[a]}$ is a semiring if and only if $\mathcal{L}_{\mathbb{H}[a]} = \mathcal{L}_{\mathbb{H}}[s(a)]$, if and only if $\nu(a) \in  \mathcal{G}(\mathbb{H})$. Moreover, as  $\bigcup_{k \in \mathbb{N} \cup \{0 \}} \mathcal{L}_{\mathbb{H}} \cdot s(a)^k \subset \mathcal{L}_{\mathbb{H}[a]}$ we have that
$\mathcal{L}_{\mathbb{H}}[s(a)]$ is the minimal semiring (with respect to inclusion) containing $\mathcal{L}_{\mathbb{H}[a]}$.
\end{cor}

\begin{cor}\label{cor2}
Let $\mathbb{H}$ be a layered semifield and $\mathbb{D}$ a layered domain extending $\mathbb{H}$. Let $a \in \mathbb{D}$ such that $\mathcal{G}(\mathbb{H}[a])$ is a  semifield. Then by Remark \ref{uniform_rem}, $\mathbb{H}[a]$ is uniform, in the sense that all layer fibers coincide. By Corollary \ref{regular_extension_cor}, the minimal uniform layered domain containing $\mathbb{H}[a]$ is $\mathcal{G}(\mathbb{H}[a])~\odot~\mathcal{L}_{\mathbb{H}}[s(a)]$.
\end{cor}

\bigbreak

%
%\begin{rem}
%Let $a \in \mathbb{D}$, and let $k \in \mathbb{N}$ the minimal number such that $a^k = \alpha \in \mathbb{H}$. Using the above bilinear representation we can rewrite this equality by $$s(a)^k \odot \nu(a)^k = (s(a) \odot \nu(a))^k = s(\alpha) \odot \nu(\alpha)$$ where $s(a) \in \mathcal{L}_{\mathbb{D}}, \nu(a) \in \mathcal{G}(\mathbb{D}), s(\alpha) \in \mathcal{L}_{\mathbb{H}}$ and $\nu(\alpha) \in \mathcal{G}(\mathbb{H})$.
%Thus we get that both $$s(a^k) = s(a)^k = s(\alpha) \in \mathcal{L}_{\mathbb{H}} \ \text{and} \  \nu(a^k) = \nu(a)^k = \nu(\alpha) \in \mathcal{G}(\mathbb{H}).$$
%Let $m_1,m_2 \in \mathbb{N}$ be the minimal numbers such that $s(a)^{m_1} = s \in \mathcal{L}_{\mathbb{H}}$ and $\nu(a)^{m_2} = \alpha \in \mathcal{G}(\mathbb{H})$ then both $m_1 | k$ and $m_2 | k$. Thus $lcm(m_1,m_2) | k$.
%Write $l=lcm(m_1,m_2) = q_1m_1 = q_2m_2$ with $q_1,q_2 \in \mathbb{N}$, then $a^l = s(a)^l \odot \nu(a)^l = s(a)^{q_1m_1} \odot \nu(a)^{q_2m_2} =  s^{q_1} \odot \alpha^{q_2} \in \mathcal{L}_{\mathbb{H}} \odot \mathcal{G}(\mathbb{H}) = \mathbb{H}$. So, by the minimality of $k$, we have that $k = lcm(m_1,m_2)$.
%\end{rem}

The above observations justify the following definition of a uniform extension of a uniform layered semifield $\mathbb{H}$, by an an arbitrary element of a layered domain \linebreak extending $\mathbb{H}$:

\begin{prop}\label{uni_ext_prop}
Let $\mathbb{H}$ be a uniform $\mathcal{L}_{\mathbb{H}}$-layered semifield. Let $\mathbb{D}$ be a uniform \linebreak $\mathcal{L}_{\mathbb{D}}$-layered domain extending $\mathbb{H}$. Then for $a \in \mathbb{D}$ we can define the following uniform extension of $\mathbb{H}$ by $a$:
\begin{align*}
\mathscr{U}_{\mathbb{D}}(\mathbb{H}[a]) = (\mathbb{H}[s(a) \odot 1])[1 \odot \nu(a)] = (\mathbb{H}[1 \odot \nu(a)])[s(a) \odot 1] = \mathbb{H}[s(a)] \cdot \mathbb{H}[\nu(a)]\\
= \{ p(s(a)) \odot \alpha \cdot s \odot \nu(q(\nu(a))) \ : \ p \in \mathcal{L}_{\mathbb{H}}[x], \ q \in \mathbb{H}[x], \ \alpha \in \mathcal{G}(\mathbb{H}), \ s \in \mathcal{L}_{\mathbb{H}} \} \\
= \mathcal{L}_\mathbb{H}[s(a)] \odot \mathcal{G}(\mathbb{H}[\nu(a)]).
\end{align*}

This is the smallest uniform layered domain extending $\mathbb{H}$ and containing $a$.
\end{prop}

In view of the above definition, we can provide a necessary and sufficient condition
for a uniform extension $\mathbb{H}[a]$ to be a semifield.
\begin{cor} \label{semifield_decomp_cor}
Let $\mathbb{H}$ be a uniform $\mathcal{L}_{\mathbb{H}}$-layered semifield. Let $\mathbb{D}$ be a uniform \linebreak $\mathcal{L}_{\mathbb{D}}$-layered domain extending $\mathbb{H}$ and let  $a \in \mathbb{D}$.
$\mathbb{H}[a]$ is a layered semifield if and only if both $\mathbb{H}[s(a)]$ and $\mathcal{G}(\mathbb{H}[\nu(a)])$
are semifields.
\end{cor}
\begin{proof}
A straightforward consequence of Proposition \ref{L_G_decomp}.
\end{proof}

\bigbreak
Let $\mathbb{H}$ be a uniform $\mathcal{L}_{\mathbb{H}}$-layered semifield. Let $\mathbb{D}$ be a uniform $\mathcal{L}_{\mathbb{D}}$-layered domain extending $\mathbb{H}$ and let $a \in \mathbb{D}$. By
Proposition \ref{uni_ext_prop}, extending $\mathbb{H}$ uniformly by $a$ involves two successive uniform extensions. One extension is an extension of the $\nu$-values semifield $\mathcal{G}(\mathbb{H})$, leaving the sorting semifield unchanged, while the other extension is an extension of the sorting semifield $\mathcal{L}_{\mathbb{H}}$, leaving the extended $\nu$-values semifield unchanged. Note that since the uniform extension is independent of the order these last two extensions are being applied, we may assume w.l.g that the $\nu$-values semifield, $\mathcal{G}(\mathbb{H})$, is extended first.\\

Write $a = s(a) \odot \nu(a)$ with $s(a) \in \mathcal{L}_{\mathbb{D}}$ and $\nu(a) \in \mathcal{G}(\mathbb{D})$. The first extension induced by $a$ is $\mathbb{H}[1 \odot \nu(a)] = \mathcal{G}(\mathbb{H}[\nu(a)])$. Now, $(\mathcal{G}(\mathbb{H}), \max, \oplus)$ is a bipotent semifield and so, we require its extension $\mathcal{G}(\mathbb{H}[\nu(a)])$ to be such too, thus this `$\nu$-values' part of the extension is just a bipotent extension. The second extension induced by $a$ is $\mathbb{H}[s(a) \odot 1] = \mathcal{L}_\mathbb{H}[s(a)]$ of $\mathcal{L}_{\mathbb{H}}$. This extension can be viewed as taking place  at each $\nu$-value in $\mathcal{G}(\mathbb{H})$. $\mathcal{L}_{\mathbb{H}}$ is a cancellative semifield and so we require its extension $\mathcal{L}_\mathbb{H}[s(a)]$ to be such too, thus this `sorting` part of the extension is a cancellative extension.

By Corollary \ref{semifield_decomp_cor} and the last paragraph, the question of when $\mathbb{H}[a]$ is a semifield reduces to the following questions: for an element $a \in \mathbb{D}$
\begin{enumerate}
  \item For which values of $s(a)$ is $\mathbb{H}[s(a)]$ a cancellative semifield extension?
  \item For which values of $\nu(a)$ is $\mathcal{G}(\mathbb{H}[\nu(a)])$ a bipotent semifield extension?
\end{enumerate}
Both of these questions have been  answered fully in the preceding sections where we characterize both bipotent (cf. Theorem \ref{mainthm}) and cancellative (cf. Corollary \ref{maincor}) simple extensions.

\bibliographystyle{amsplain}
\bibliography{TAGBib}

\end{document}